\documentclass[11pt,a4paper,reqno]{amsart}
\usepackage{amssymb,amsmath,amsthm}
\usepackage[utf8]{inputenc}
\usepackage[T1]{fontenc}
\usepackage{enumerate}
\usepackage[all]{xy}
\usepackage{fullpage}
\usepackage{comment}
\usepackage{stmaryrd}
\usepackage{array}
\usepackage{hyperref}

\def\Z{\mathbb{Z}}
\def\N{\mathbb{N}}
\def\Q{\mathbb{Q}}
\def\oQ{\overline{\mathbb{Q}}}

\def\Gal{\mathrm{Gal}}
\def\res{\mathrm{res}}
\def\Aut{\mathrm{Aut}}

{\theoremstyle{plain}
\newtheorem{theorem}{Th\'eor\`eme}[section]
\newtheorem{proposition}[theorem]{Proposition}
\newtheorem{lemma}[theorem]{Lemme}
\newtheorem{corollary}[theorem]{Corollaire}

}

{\theoremstyle{remark}
\newtheorem{definition}[theorem]{Définition}}

\title{R\'{e}alisations galoisiennes explicites de certaines familles de $2$-groupes}
\author{Angelot BEHAJAINA}
\email{angelot.behajaina@unicaen.fr}
\address{Laboratoire de Mathématiques Nicolas Oresme, Université de Caen Normandie, BP 5186, 14032 Caen Cedex, France}
\newenvironment{poliabstract}[1]
{\renewcommand{\abstractname}{#1}\begin{abstract}}
{\end{abstract}}
\begin{document}

\maketitle

\begin{abstract}
In this paper, we construct, for some $2$-groups $G$, explicit Galois extensions $E/\Q(T)$ of group $G$ with $E \cap \oQ=\Q$. We also provide explicit arithmetic progressions of integers $t_{0}$ such that the specialization $E_{t_0}/\Q$ of $E/\Q(T)$ at $t_0$ has Galois group $G$.
\end{abstract}

\begin{poliabstract}{R\'esum\'e}
Dans cet article, nous construisons, pour certains $2$-groupes $G$, des extensions galoisiennes explicites $E/\Q(T)$ de groupe $G$ vérifiant $E \cap \oQ=\Q$. Nous fournissons aussi des progressions arithmétiques explicites d'entiers $t_0$ telles que la spécialisation $E_{t_0}/\Q$ de $E/\Q(T)$ en $t_0$ soit de groupe $G$. 
\end{poliabstract}

\section{Introduction} \label{sec:intro}
Le {\it problème inverse de Galois} consiste à savoir si tout groupe fini se réalise comme le groupe de Galois d'une extension galoisienne $F/\Q$. Bien que remontant à Hilbert et Noether, et malgré de nombreux travaux, ce célèbre problème de la théorie des nombres demeure largement ouvert. L'approche géométrique à ce problème consiste à montrer que tout groupe fini $G$ est {\it{groupe de Galois régulier sur $\Q$}}, c'est-à-dire à introduire une indéterminée $T$ et à construire une extension galoisienne $E/\Q(T)$ de groupe $G$ telle que $E/\Q$ soit {\it régulière} (c'est-à-dire $E \cap \oQ =\Q$). Le {\it théorème d'irréductibilité de Hilbert} fournit alors une infinité de nombres rationnels $t_0$ tels que la {\it spécialisation} $E_{t_0}/\Q$ de $E/\Q(T)$ en $t_0$ soit galoisienne de groupe $G$ (voir \S\ref{ssec:corps} pour la définition). Certains groupes finis, simples non abéliens notamment, ont été réalisés par cette méthode. On renvoie aux ouvrages de référence \cite{Ser92}, \cite{Vol96}, \cite{FJ08} et \cite{MM18} pour plus de détails.

Une version raffinée du problème inverse de Galois et de sa version régulière consiste, pour un groupe fini $G$ donné, à construire une extension galoisienne ou un polynôme explicite de groupe $G$. Par exemple, pour $G=\Z/2\Z$, on peut prendre $F/\Q=\Q(\sqrt{2})/\Q$. Un autre exemple est le trinôme $Y^n-Y-T$ $(n \geq 3$) dont le corps de décomposition $E$ sur $\Q(T)$ vérifie $\Gal(E/\Q(T))=S_{n}$ et $E \cap \oQ=\Q$ (voir \cite[\S4.4]{Ser92}). On renvoie aux dernières pages de \cite{MM18} pour d'autres réalisations explicites de certains groupes finis $G$, en général de petit cardinal.

Dans cet article, nous nous intéressons à certains $2$-groupes, plus précisément aux groupes non abéliens d'ordre $2^{n}$ et d'exposant $2^{n-1}$. Pour $n \geq 3$, il existe exactement quatre tels groupes (voir \cite[page 127]{JLY02}), à savoir le {\it groupe diédral $D_{2^{n}}$}, le {\it groupe quasi-diédral $QD_{2^n}$}, le {\it groupe modulaire $M_{2^n}$} et le {\it groupe des quaternions généralisés $Q_{2^{n}}$} (voir \S\ref{ssec:groupes} pour une présentation de chacun de ces groupes). Bien entendu, ces groupes étant résolubles, ils sont groupes de Galois sur $\Q$ par le théorème de Shafarevich (voir \cite[(9.6.1)]{NSW08}). Ces groupes sont en fait groupes de Galois réguliers sur $\Q$. En effet, il est connu que tout produit en couronnes $\Z/m\Z \wr \Z/h\Z$ est groupe de Galois d'une extension galoisienne $E/\Q(T_1, \dots, T_s)$ vérifiant $E \cap \overline{\Q}=\Q$ pour un certain $s \geq 1$, qui est en fait égal à $h$, et que tout produit semi-direct $\Z/m\Z \rtimes \Z/h\Z$ est quotient de $\Z/m\Z \wr \Z/h\Z$ (voir \cite[\S16.4]{FJ08}). Un argument de spécialisation (voir \cite[Proposition 13.2.1]{FJ08}), en général non explicite, permet alors de prendre $s=1$ (pour $h$ quelconque). Ceci s'applique en particulier aux 2-groupes ci-dessus puisque les trois premiers sont des produits semi-directs $\Z/2^{n-1}\Z\rtimes \Z/2\Z$ et le dernier est quotient d'un produit semi-direct $\Z/2^{n-1}\Z \rtimes \Z/4\Z$.

Par contre, l'existence de réalisations explicites de ces groupes est en général inconnue, notamment pour les groupes de quaternions généralisés (voir \cite[page 141]{JLY02}). Dans la suite, nous construisons de telles extensions en développant et en rendant explicite la méthode de \cite[\S16.4]{FJ08}. Nous montrons par exemple le théorème suivant (voir théorème \ref{quaternion}) :
\begin{theorem}\label{quaternions}
Soient $n \geq 3$ et $\xi=\exp( 2 \pi i/2^{n-1})$. Pour $l \in \llbracket 1, 2^{n-1}\rrbracket$ et $\ell \in \{1,2\}$, on pose 
\[
z_{\ell,l}=\sum_{j \in (\Z/2^{n-1}\Z)^{*}}\xi^{lj}\prod_{k \in (\Z/2^{n-1}\Z)^{*}} ( T+1+(-1)^{\ell}\sqrt{T^{2}+1}-\xi^{k} )^{r(j/k)/2^{n-1}},
\]
o\`{u} $r:(\Z/2^{n-1}\Z)^{*}\rightarrow \llbracket 0, 2^{n-1}-1 \rrbracket$ envoie $k \in (\Z/2^{n-1}\Z)^{*}$ sur son unique repr\'{e}sentant modulo $2^{n-1}$. Si l'on note
\[
w= \sqrt{T^{2}+1}+(\sqrt{T^2+1+T\sqrt{T^{2}+1}}+\sum_{l=1}^{2^{n-1}}z_{1,l}z_{2,l})^{2},
\] alors l'extension $\Q(T,w)/\Q(T)$
est galoisienne de groupe $Q_{2^{n}}$ et $\Q(T,w)/\Q$ est r\'{e}guli\`{e}re. De plus, les $\Q(T)$-conjugu\'{e}s de $w$ sont les 
\[(-1)^{a}\sqrt{T^{2}+1}+(\pm \sqrt{T^2+(-1)^{a}T\sqrt{T^{2}+1}}+\sum_{l=1}^{2^{n-1}}z_{2,l}z_{1,l+s})^{2}, \quad (a,s) \in \{ 0,1 \} \times \llbracket 0,2^{n-2}-1 \rrbracket. 
\]
\end{theorem}

\noindent
Nous donnons aussi des analogues pour $D_{2^n}$, $QD_{2^n}$ et $M_{2^n}$ (voir corollaires \ref{princ1}, \ref{princ2} et \ref{princ3}).

Nous construisons en fait une réalisation régulière explicite de n'importe quel produit semi-direct $\Z/m\Z \rtimes \Z/2\Z$ ($m \geq 3$), ce qui généralise nos résultats sur $D_{2^n}$, $QD_{2^n}$ et $M_{2^n}$ (voir théorème \ref{princ}). Ce résultat plus général nous permet aussi de construire une réalisation régulière explicite du groupe diédral $D_{2m}$ à $2m$ éléments ($m \geq 3$), voir corollaire \ref{princ1}. Ceci fournit une variante régulière, valable pour tout $m$, d'une construction de Martinais et Schneps (voir \cite{MS92}). 

Notons toutefois que notre méthode diffère de celle de \cite[\S16.4]{FJ08} puisque nous cons\-truisons "directement" des extensions galoisiennes $E/\Q(T)$ de groupe $\Z/m\Z \rtimes \Z/2\Z$, c'est-à-dire notre méthode ne nécessite pas l'argument de spécialisation nécessaire pour passer de $\Q(T_1, T_2)$ à $\Q(T)$ rappelé plus haut. De plus, la stratégie de la preuve du théorème \ref{quaternions} ne consiste pas à réaliser explicitement $\Z/2^{n-1}\Z \wr \Z/4\Z$, puis ses quotients $\Z/2^{n-1}\Z \rtimes \Z/4\Z$, puis $Q_{2^n}$. Elle consiste plutôt à remarquer que $Q_{2^n}$ est quotient d'un produit semi-direct $\Z/2^{n-1}\Z \rtimes \Z/4\Z$ qui est en fait un produit fibré de $\Z/4\Z$ et de $D_{2^n}$. Nous utilisons alors la réalisation régulière explicite de $D_{2^n}$ obtenue dans le corollaire \ref{princ1} pour réaliser explicitement ce produit fibré, et donc $Q_{2^n}$.

Nous donnons ensuite des progressions arithmétiques explicites d'entiers $t_0$ tels que la spécialisation de $E/\Q(T)$ en $t_0$ soit galoisienne de groupe $Q_{2^n}$, où $E/\Q(T)$ désigne l'extension de groupe $Q_{2^n}$ construite dans le théorème \ref{quaternions}. L'existence de progressions arithmétiques d'entiers satisfaisant à la propriété de spécialisation de Hilbert a été étudiée par de nombreux auteurs, par exemple, par Davenport--Lewis--Schinzel (voir \cite{Sch00}), Fried (voir \cite{Fri74}), Dèbes--Ghazi (voir \cite{DG12}), Dèbes--Legrand (voir \cite{DL13}) et Legrand (voir \cite{Leg16}). Nous explicitons la méthode de \cite{Leg16}, qui repose sur l'inertie des spécialisations, et obtenons le théorème suivant : 

\begin{theorem}\label{thspecarith}
Soit $n \geq 3$. Soient $p$ et $q$ deux nombres premiers distincts supérieurs ou égaux à $7^{2^{n-2}}+1$ tels que $p \equiv 1 \pmod {2^{n-1}}$ et $q \equiv 1 \pmod 4$. Alors il existe un $t_{0} \in \llbracket 0, p^{2}q^{2}-1\rrbracket$ explicite tel que, si $t$ désigne n'importe quel entier positif vérifiant $t \equiv t_{0} \pmod{p^{2}q^{2}}$, alors la spécialisation $E_{t}/\Q$ de $E/\Q(T)$ en $t$ est galoisienne de groupe $Q_{2^{n}}$.
\end{theorem}
\noindent
Nous renvoyons au théorème \ref{thm:spec_arith} pour un énoncé plus général où l'on explique comment cons\-truire un tel $t_0$. Par ailleurs, il nous paraît plausible que des analogues peuvent être donnés pour les autres $2$-groupes considérés dans cet article. Nous laissons ce travail au lecteur intéressé.

L'article est organis\'{e} de la mani\`{e}re suivante. La section 2 est dédiée aux pr\'{e}liminaires. Dans la section 3, nous construisons des réalisations régulières explicites de n'importe quel produit semi-direct $\Z/m\Z \rtimes \Z/2\Z$ ($m \geq 3$). La section 4 est non seulement dédiée à la même tâche pour les groupes de quaternions généralisés mais aussi à la construction de réalisations explicites de ces groupes sur $\Q$ par spécialisation.

{\bf Remerciements.---} Je tiens à remercier mes directeurs de thèse, Bruno Deschamps et François Legrand, pour leurs nombreuses relectures, commentaires utiles et précieuses suggestions. Je remercie également le GDRI GANDA pour son soutien financier et Denis Simon pour certaines discussions lors de la préparation de ce travail.

\section{Pr\'{e}liminaires}\label{prel}

\subsection{Extensions de corps de fonctions} \label{ssec:corps}

Etant donnés un groupe fini $G$ et un corps $K$, une {\it $G$-extension} de $K$ est une extension galoisienne $F/K$ de groupe $G$. Si $F/K$ et $M/L$ désignent deux extensions galoisiennes finies telles que $K \subset L$ et $F \subset M$, l'application de restriction $\Gal(M/L) \rightarrow \Gal(F/K)$ sera notée $\res^{M/L}_{F/K}$.

Etant donnée une indéterminée $T$, on dit qu'une extension finie galoisienne $E/\Q(T)$ est {\it$\Q$-régulière} si $E/\Q$ est régulière, c'est-à-dire si $E \cap \overline{\Q}=\Q$. On dit que $t_0 \in \mathbb{P}^1(\overline{\Q})$ est un {\it point de branchement} de $E/\Q(T)$ si l'idéal $\langle T-t_0 \rangle$  est ramifié dans la clôture intégrale de $\overline{\Q}[T-t_0]$ dans $E\overline{\Q}$ (si $t_0=\infty$, $T-t_0$ doit être remplacé par $1/T$). 
 
On suppose maintenant que $E/\Q(T)$ est une $G$-extension $\Q$-régulière de points de branchement $t_{1},\dots, t_{r}$. Etant donné $t_{0} \in \mathbb{P}^{1}(\Q)\setminus \{t_1, \dots, t_r \}$, la {\it spécialisation} de $E/\Q(T)$ en $t_0$, notée $E_{t_0}/\Q$, est l'extension résiduelle en un idéal premier $\mathcal{P}$ au dessus de $\langle T-t_0 \rangle$. Comme $E/\Q(T)$ est galoisienne, l'extension $E_{t_0}/\Q$ ne dépend pas du choix de l'idéal premier $\mathcal{P}$. De plus, $E_{t_0}/\Q$ est galoisienne et son groupe de Galois est le groupe de décomposition de $E/\Q(T)$ en $\mathcal{P}$.

Etant donn\'{e} un nombre premier $p$, on pose $1/\infty = 0$, $1 / 0 = \infty$, $v(\infty) = -\infty$ et $v(0) = \infty$, où $v$ désigne la valuation $p$-adique. Soit $\Z_{p\Z}$ le localisé de $\Z$ en $p$. Pour chaque $t \in \mathbb{P}^{1}(\overline{\Q})$, on note $m_{t}(X) \in \Q[X]$ le polynôme minimal de $t$ sur $\Q$, avec la convention $m_{\infty}(X)=1$. 
\begin{definition} \label{rencontre} 
Soient $t_0$ et $t_1$ dans $\mathbb{P}^1(\overline{\Q})$. On dit que {\it $t_0$ et $t_1$ se rencontrent modulo $p$} s'il existe un corps de nombres $F$ et un idéal premier de $F$ au dessus de $p$ de valuation associée $w$ tels que $t_0$, $t_1$ $\in \mathbb{P}^1(F)$ et tels que l'une des conditions suivantes soit vérifiée : 
\vskip 0.5mm
\noindent
1) $w(t_0) \geq 0$, $w(t_1) \geq 0$ et $w(t_0-t_1) > 0$,
\vskip 0.5mm
\noindent
2) $w(t_0) \leq 0$, $w(t_1) \leq 0$ et $w((1/t_0) - (1/t_1)) > 0$.
\end{definition}

Le lemme suivant nous sera utile par la suite.

\begin{lemma}\label{normren}
Soient $t_0$ et $t_1$ dans $\overline{\Q}$ entiers sur $\Z_{p\Z}$. Si $t_0$ et $t_1$ se rencontrent modulo $p$ et si $N(t_0-t_1)$ désigne la norme de $t_0-t_1$ dans une extension finie galoisienne $F/\Q$ donnée telle que $t_0, t_1 \in F$, alors $v(N(t_0-t_1)) >0$.
\end{lemma}

\begin{proof}[Preuve]
On se donne un idéal premier de $F$ au dessus de $p$ de valuation associée $w$ tel que $w(t_0-t_1) > 0$. On a $w(N(t_0-t_1))=\sum_{\sigma \in \Gal(F/\Q)}w(\sigma(t_0-t_1)) \geq w(t_0-t_1) >0$.
\end{proof}

\begin{definition}\label{bonpremiers} On dit que $p$ est un \textit{mauvais premier} (et un \textit{bon premier} sinon) de l'extension $E/\Q(T)$ si l'une des conditions suivantes est vérifiée :
\vskip 0.5mm
\noindent
1) $|G| \in p\Z$,
\vskip 0.5mm
\noindent
2) deux points de branchement distincts de $E/\Q(T)$ se rencontrent modulo $p$,
\vskip 0.5mm
\noindent
3) $p$ est \textit{verticalement ramifié} dans $E/\Q(T)$, c'est-à-dire l'idéal $p\Z[T]$ se ramifie dans la clôture intégrale de $\Z[T]$ dans $E$,
\vskip 0.5mm
\noindent
4) $p$ se ramifie dans $\Q(t_{1},\dots,t_{r})/\Q$, où $t_{1}, \dots, t_{r}$ sont les points de branchement de $E/\Q(T)$. 
\end{definition}

A chaque point de branchement $t_i$ est associée une classe de conjugaison $C_i$ de $G$, appelée {\it classe canonique de l'inertie}. En effet, les groupes d'inertie de $E\overline{\Q}/\overline{\Q}(T)$ en $t_i$ sont des groupes cycliques deux à deux conjugués et d'ordre égal à l'indice de ramification $e_{i}$. De plus, chacun d'eux admet un générateur distingué correspondant à l'automorphisme $(T-t_{i})^{1/e_{i}} \mapsto \mathrm{exp}(2 \pi i /e_{i}) (T-t_{i})^{1/e_{i}}$ de $\overline{\Q}(((T-t_{i})^{1/e_{i}}))$ (on remplace $T-t_{i}$ par $1/T$ si $t_{i}=\infty$). Alors $C_i$ est la classe de conjugaison de tous les générateurs distingués des groupes d'inertie en $t_{i}$. Le $r$-uplet non ordonné $(C_1,\dots,C_r)$ est appelé {\it invariant canonique de l'inertie} de $E/\Q(T)$. Pour $i \in \{ 1, \dots, r \}$, notons $g_{i}$ le générateur distingué d'un certain groupe d'inertie de $E\overline{\Q}/\overline{\Q}(T)$ en $t_{i}$.  

Nous renvoyons à \cite[Proposition 4.2]{Bec91} et \cite[\S 2.2.3]{Leg16} pour le théorème suivant.
\begin{theorem}\label{spinth} Soit $t_{0} \in \mathbb{P}^{1}(\Q) \setminus \{t_{1}, \dots, t_{r}\}$. Fixons $j \in \llbracket 1, r \rrbracket$ tel que $t_0$ et $t_{j}$ se rencontrent modulo $p$. Supposons que $p$ soit un bon premier pour $E/\Q(T)$ et que $m_{t_j}(T)$ et $m_{1/t_{j}}(T)$ soient dans $\Z_{p\Z}[T]$. Alors le groupe d'inertie de $E_{t_0}/\Q$ en $p$ est conjugué dans $G$ à $\langle g_j^a \rangle$, où $a = v(m_{t_{j}}(t_0))$ (resp. $a=v(m_{1/t_{j}}(1/t_0))$) si $v(t_0)\geq 0$ (resp. $v(t_0)\leq 0$).   
\end{theorem}

\subsection{Sur les $2$-groupes} \label{ssec:groupes}
Dans cet article, le groupe cyclique d'ordre $m$ est noté $\Z/m\Z$ et considéré additivement.

Un produit semi-direct $\Z/m\Z \rtimes \Z/2\Z$ est déterminé par l'image $d$ de $1 \in \Z/2\Z$ dans $\Aut(\Z/m\Z)=(\Z/m\Z)^{*}$ (on a nécessairement $d^2=1$). Une présentation de ce groupe est
\begin{equation}\label{predih}
\Z/m\Z \rtimes \Z/2\Z = \langle r , s \mid r ^{m}=s^{2}=1, s r s^{-1} =r^{d} \rangle,
\end{equation} où  $r$ (resp. $s$) correspond à $(1,0)$ (resp. $(0,1)$).  Dans la suite, on considérera trois cas :
\begin{enumerate}[1)]
\item $m$ arbitraire et $d=-1$; dans ce cas, $\Z/m\Z \rtimes \Z/2\Z$ est le groupe diédral $D_{2m}$,
\item $m=2^{n-1}$ ($n \geq 3$) et $d=2^{n-2}-1$; dans ce cas, le produit semi-direct $\Z/2^{n-1}\Z \rtimes\Z/2\Z$ est le groupe quasi-diédral $QD_{2^n}$,
\item $m=2^{n-1}$ ($n \geq 3$) et $d=2^{n-2}+1$; dans ce cas, le produit semi-direct $\Z/2^{n-1}\Z \rtimes \Z/2\Z$ est le groupe modulaire $M_{2^n}$.
\end{enumerate}  

Etant donné $n \geq 3$, on définit le groupe $\Delta_{n}$ comme étant le produit semi-direct $\Z/2^{n-1}\Z \rtimes \Z/4\Z$ où l'image de $1 \in \Z/4\Z$ dans $\Aut(\Z/2^{n-1}\Z)$ est $-\mathrm{Id}$. Ce groupe admet comme présentation
\begin{equation}\label{presdelt}
\Delta_{n}= \langle \rho,{\sigma} \mid {\rho}^{2^{n-1}}={\sigma}^{4}=1,{\sigma} {\rho} {\sigma}^{-1}={\rho}^{-1} \rangle,
\end{equation} où ${\rho}$ (resp. ${\sigma}$) correspond à $(1,0)$ (resp. $(0,1)$). Le groupe des quaternions généralisés d'ordre $2^{n}$, noté  $Q_{2^{n}}$, est le quotient de $\Delta_{n}$ par le sous-groupe distingué $\langle(2^{n-2},2) \rangle$. La présentation (\ref{presdelt}) de $\Delta_{n}$ induit la présentation de $Q_{2^{n}}$ suivante :
\begin{equation}\label{presquat}
Q_{2^{n}}= \langle \Lambda, \Sigma \mid \Lambda^{2^{n-1}}=\Sigma^{4}=1,\Lambda^{2^{n-2}}=\Sigma^{2},\Sigma \Lambda \Sigma^{-1}=\Lambda^{-1} \rangle,
\end{equation} où $\Lambda$ (resp. $\Sigma$) est l'image de ${\rho}$ (resp. ${\sigma}$) modulo $\langle(2^{n-2},2) \rangle$. 

Pour $m \geq 2$, le produit en couronnes de $\Z/m\Z$ et $\Z/2\Z$, noté $\Z/m\Z \wr \Z/2\Z$, est le produit semi-direct $(\Z/m\Z \times \Z/m\Z) \rtimes \Z/2\Z$ où l'image de $1 \in \Z/2\Z$ dans $\Aut(\Z/m\Z \times \Z/m\Z)$ est l'automorphisme $(a,b) \mapsto (b,a)$. 

Dans ce texte, on utilise toujours les présentations ci-dessus.

La proposition suivante nous sera utile par la suite. 

\begin{proposition}\label{propgen}
Soient $n \geq 3$, $K$ un corps et $M/K$ une $D_{2^n}$-extension. Notons $L$ le sous-corps de $M$ fixé par ${r}$, où ${r}$ est défini dans la présentation \eqref{predih}.  Supposons que $L/K$ se plonge dans une $\Z/4\Z$-extension $H/K$ et soit $\tau$ un générateur de $\Gal(H/K)$. D'une part, $H/L$ et $M/L$ sont linéairement disjointes. D'autre part, $HM/K$ est une $\Delta_{n}$-extension. De plus, il existe un unique relèvement ${\rho}$ de ${r}$ dans $\Gal(HM/H	)$ et celui-ci vérifie ce qui suit : pour tout relèvement ${\sigma}$ de $\tau$ dans $\Delta_{n}$, les éléments ${\rho}$ et ${\sigma}$ vérifient la présentation \eqref{presdelt}.
\end{proposition}
\begin{proof}[Preuve]
Tout d'abord, il est clair que $HM/K$ est galoisienne. De plus, on a $H \not\subset M$ car $\Z/4\Z$ n'est pas quotient de $D_{2^{n}}$. Par conséquent, on a $[HM:M]=2$, les extensions $H/L$ et $M/L$ sont linéairement disjointes et $\res^{HM/H}_{M/L}$ est un isomorphisme. Ainsi, ${r}$ admet un unique relèvement ${\rho}$ dans $\Gal(HM/H)$. A partir de maintenant, on se donne un relèvement ${\sigma}$ de $\tau$ \`a $HM$. 

Maintenant, ${\sigma}$ est d'ordre $4$. En effet, on a $o({\sigma}) \geq o(\tau)=4$. De plus, ${\sigma}$ ne fixe pas $L=M^{\langle {r} \rangle}$ et ${\rm{res}}^{HM/K}_{M/K}({\sigma}) \in D_{2^n}$. Par conséquent, on a ${\rm{res}}^{HM/K}_{M/K}({\sigma}) ^2 =\mathrm{Id}_{M}$. Puisque $[HM:M]=2$, on obtient $o({\sigma}^{2}) \leq 2$. Ainsi $o({\sigma}) \leq 4$ et donc $o({\sigma})=4$.

Ensuite, ${\sigma}{\rho} {\sigma}^{-1} = {\rho}^{-1}$. En effet, on a $\mathrm{res}^{HM/H}_{M/L}({\sigma} {\rho} {\sigma}^{-1})= \mathrm{res}^{HM/K}_{M/K}({\sigma}) \cdot {r} \cdot \mathrm{res}^{HM/K}_{M/K}({\sigma})^{-1}={r}^{-1}$
car $\mathrm{res}^{HM/K}_{M/K}({\sigma}) \notin \langle {r} \rangle$. Or ${r}^{-1}=\mathrm{res}^{HM/H}_{M/L}({\rho}^{-1})$ et $\mathrm{res}^{HM/H}_{M/L}$ est injective. Donc ${\sigma}{\rho} {\sigma}^{-1} ={\rho}^{-1}$.
 
De plus, on a $\Gal(HM/K)=\langle {\rho},{\sigma} \rangle$. En effet, soit $s$ comme dans \eqref{predih}. Le morphisme de restriction
$\langle{\rho},{\sigma} \rangle \rightarrow  \langle {r}, {s} \rangle =D_{2^{n}}$ est surjectif puisque l'on a ${s}={r}^{a}\res^{HM/K}_{M/K}({\sigma})$ pour un certain $a$. Ainsi $2^{n}=|D_{2^{n}}|$ divise $|\langle {\rho},{\sigma} \rangle |$, donc divise $|\Gal(HM/K)|= 2^{n+1}$. Le morphisme n'est pas injectif car $o({\sigma})=4 \neq 2= o({r}^{-a}{s})=o(\res^{HM/K}_{M/K}({\sigma}))$. On en d\'{e}duit donc $\Gal(HM/K)=\langle {\rho},{\sigma} \rangle$.

Enfin, on a $\Gal(HM/K)=\Delta_{n}$ et ${\rho}$, ${\sigma}$ v\'{e}rifient  (\ref{presdelt}). En effet, les \'{e}galit\'{e}s ${\rho}^{2^{n-1}}=1$, ${\sigma}^{4}=1$, ${\sigma} {\rho}{\sigma}^{-1} ={\rho}^{-1}$ et $|\langle{\rho},{\sigma} \rangle |=2^{n+1}$ montrent que l'on a $\langle {\rho},{\sigma} \rangle= \Delta_{n}$. Le paragraphe précédent donne alors $\Gal(HM/K)=\Delta_{n}$.
\end{proof}
La proposition pr\'{e}c\'{e}dente admet la version r\'{e}guli\`{e}re suivante :
\begin{proposition}\label{regDelta}
Soient $n \geq 3$ et $M/\Q(T)$ une $D_{2^{n}}$-extension $\Q$-r\'{e}guli\`{e}re. Notons $L$ le sous-corps de $M$ fix\'{e} par ${r}$, o\`u ${r}$ est défini dans \eqref{predih}. Supposons que $L/\Q(T)$ se plonge dans une $\Z/4\Z$-extension $H/\Q(T)$. Alors $HM/\Q(T)$ est une $\Delta_{n}$-extension $\Q$-r\'{e}guli\`{e}re.
\end{proposition}
\begin{proof}[Preuve]
Supposons que $H/\Q$ ne soit pas régulière. Alors $[H \cap \overline{\Q}:\Q]=2$ car $L \subset H$ et $L/\Q(T)$ est une extension de degr\'{e} $2$ $\Q$-r\'{e}guli\`{e}re. Ainsi, $(H \cap \overline{\Q})(T)/\Q(T)$ et $L/\Q(T)$ sont lin\'{e}airement disjointes, ce qui entraîne $\Gal(H/\Q(T)) \simeq \Z/2\Z \times \Z/2\Z$, une contradiction. Par cons\'{e}quent, $H/\Q(T)$ est $\Q$-r\'{e}guli\`{e}re. En appliquant la proposition \ref{propgen} aux extensions $M/\Q(T)$ et $H/\Q(T)$ (resp. $M\overline{\Q}/\overline{\Q}(T)$ et $H\overline{\Q}/\overline{\Q}(T)$), on obtient que $HM/\Q({T})$ (resp. $HM\overline{\Q}/\overline{\Q}(T)$) est une $\Delta_{n}$-extension. Ainsi, $HM/\Q(T)$ est $\Q$-r\'{e}guli\`{e}re.
\end{proof}

\subsection{Un lemme sur les racines de l'unité} 
Nous terminons cette section avec un lemme, que nous utiliserons à plusieurs reprises par la suite :
\begin{lemma}\label{ramggg}
Soient $m \geq 3$ et $\xi=\exp(2 \pi i /m)$. Posons $s_{k}=-(1-\xi^{k})/2+1/(2(1-\xi^{k}))$ pour tout $k \in (\Z/m\Z)^{*}$. Soit $\overline{\mathcal{O}}$ la clôture intégrale de $\oQ[T]$ dans $\oQ(T)(\sqrt{T^2+1})$.
\vskip 0.5mm
\noindent
a) Pour $k \neq k'$, on a $s_{k} \neq s_{k'}$.
\vskip 0.5mm
\noindent
b) Pour $\epsilon \in \{-1,1\}$ et $k \in (\Z/m\Z)^{*}$, on a  
$s_{k} \neq \epsilon i$.
\vskip 0.5mm
\noindent
c) Les id\'{e}aux $(T+1 +(-1)^{\ell} \sqrt{T^{2}+1}- \xi^{k})\overline{\mathcal{O}}$ ($k \in (\Z/m\Z)^{*}$ et $\ell \in \{0,1\}$) sont premiers et deux \`{a} deux distincts. De plus, pour $k \in (\Z/m\Z)^{*}$, on a $(T+1 \pm \sqrt{T^{2}+1}- \xi^{k})\overline{\mathcal{O}} \cap \oQ[T]= (T-s_k)\oQ[T]$.
\end{lemma}
\begin{proof}[Preuve] 
a) S'il y avait \'{e}galit\'{e}, on aurait $(1-\xi^{k})(1-\xi^{k'})=-1$ et donc $
\xi^{k}=(\xi^{k'}-2)/(\xi^{k'}-1)$. Ainsi $\xi^{k'}$ serait à la fois sur le cercle unité et la médiatrice du segment $[1,2]$, une contradiction.  
\vskip 1mm
\noindent
b) S'il y avait \'{e}galit\'{e}, on aurait $\xi^{2k}-2(1+i\epsilon)\xi^{k}+2i\epsilon=0 $ et donc $\xi^{k}=1+i\epsilon$, une contradiction.
\vskip 1mm
\noindent 
c) On a tout d'abord
$(T+1 + \sqrt{T^{2}+1}- \xi^{k} )(T+1 - \sqrt{T^{2}+1}- \xi^{k})=2(1-\xi^k)(T-s_k)$, 
ce qui donne $(T+1 + \sqrt{T^{2}+1}- \xi^{k} )\overline{\mathcal{O}}\cdot (T+1 - \sqrt{T^{2}+1}- \xi^{k} )\overline{\mathcal{O}}=(T-s_k)\overline{\mathcal{O}}$.
Puisque $(T-s_k)\overline{\Q}[T] \neq \overline{\Q}[T]$, on a $(T-s_k)\overline{\mathcal{O}} \not= \overline{\mathcal{O}}$. Par conséquent, $T+1 + \sqrt{T^{2}+1}- \xi^{k} $ et son conjugué $T+1 - \sqrt{T^{2}+1}- \xi^{k}$ ne sont pas inversibles dans $\overline{\mathcal{O}}$. Puisque $[\oQ(T)(\sqrt{T^2+1}):\overline{\Q}(T)]=2$,  on en déduit que $(T+1 + \sqrt{T^{2}+1}- \xi^{k} )\overline{\mathcal{O}}$ et $(T+1 - \sqrt{T^{2}+1}- \xi^{k} )\overline{\mathcal{O}}$ sont des idéaux premiers. De plus, ceux-ci sont nécessairement distincts en vertu de \cite[Proposition 6.2.3]{Sti09} et du b). Enfin, le a) entraîne que les idéaux $(T-s_k)\overline{\Q}[T]$ ($k \in (\Z/m\Z)^{*}$) sont deux à deux distincts.
\end{proof}
\section{Produits semi-directs $\Z/m\Z \rtimes \Z/2\Z$}\label{realexpl1}
Dans cette partie, nous construisons, pour tout $m \geq 3$ et tout produit semi-direct $G=\Z/m\Z \rtimes \Z/2\Z$, une extension galoisienne $\Q$-régulière explicite de $\Q(T)$ de groupe $G$ (voir théorème \ref{princ}). 

\subsection{Notations}\label{notprin}
On pose $L=\Q(T)(\sqrt{T^{2}+1})$ et on note $\beta$ le générateur de $\Gal(L/\Q(T))$. On note $\mathcal{O}$ (resp. $\overline{\mathcal{O}}$) la clôture intégrale de $\Q[T]$ (resp. $\oQ[T]$) dans $L$ (resp. $L\oQ$). On se donne $m \geq 3$ et on note $\xi=\exp(2 \pi i/m)$. Pour $\ell \in \{1,2 \}$ et $k \in (\Z/m\Z)^{*}$, on note $\mathcal{P}_{\ell,k}$ l'idéal premier de $\overline{\mathcal{O}}$ engendré par $T+1+(-1)^{\ell}\sqrt{T^{2}+1}-\xi^{k}$ (voir lemme \ref{ramggg}) et on choisit une racine $m$-i\`{e}me $x_{\ell,k}$ de $T+1+(-1)^{\ell}\sqrt{T^{2}+1}-\xi^{k}$ dans $\overline{\Q(T)}$. Pour $ \ell \in \{1,2\}$, on note $M_\ell$ le compositum des corps $L(\xi, x_{\ell, k})$ $(k \in (\Z/m\Z)^*$) et, pour tout $j \in (\Z/m\Z)^{*}$, on pose

\[
y_{\ell,j}=\prod_{k \in (\Z/m\Z)^{*}}x_{\ell,k}^{r_{m}(j/k)},
\] o\`{u} $r_{m}:(\Z/m\Z)^{*}\rightarrow \llbracket 0, m-1\rrbracket$ envoie $k \in (\Z/m\Z)^{*}$ sur son unique repr\'{e}sentant modulo $m$. Pour tout $\ell \in \{1,2\}$ et tout $l \in \N$, on pose
\[
z_{\ell,l}=\sum_{j \in (\Z/m\Z)^{*}}\xi^{lj}y_{\ell,j}=\sum_{j \in (\Z/m\Z)^{*}}\xi^{lj}\prod_{k \in (\Z/m\Z)^{*}}x_{\ell,k}^{r_{m}(j/k)}.
\] 
Pour tout $\ell \in \{1,2\}$, notons $N_{\ell}=L(z_{\ell,1})$ et
$
h_{\ell}(X)= \prod_{l=1}^{m}(X-z_{\ell,l}).
$
Posons $F=N_{1}N_{2}=L(z_{1,1},z_{2,1})$ et $E_{d}=\Q(T,\sqrt{T^{2}+1}+\sum_{l=1}^{m}z_{2,l}z_{1,-dl})$ pour tout $d \in (\Z/m\Z)^{*}$ vérifiant $d^2=1$. 

On a le diagramme suivant :
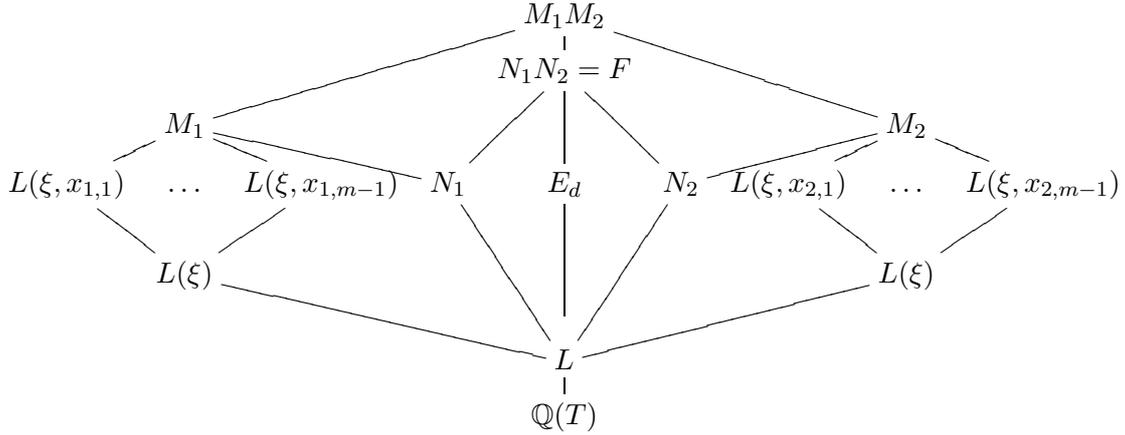
\[{\xymatrix@-4ex{ & & & & M_1M_2 \ar@{-}[d] & & & & \\
& & & & N_{1}N_{2}=F  \ar@{-}[dd] & & & & \\
& M_1  \ar@{-}[rrruu] \ar@{-}[dr] \ar@{-}[drr] & & & & & &  \ar@{-}[llluu] M_2  \ar@{-}[dll]  \ar@{-}[dr] & \\
L(\xi ,x_{1,1})  \ar@{-}[ru] & \dots & L(\xi, x_{1,{m-1}}) & N_{1}  \ar@{-}[ddddr] \ar@{-}[ruu] &E_{d} \ar@{-}[ddd] & \ar@{-}[luu] N_{2}  \ar@{-}[ddddl] & L(\xi, x_{2,1}) \ar@{-}[ru] & \dots & L(\xi, x_{2,{m-1}}) \\
& & & &  \ar@{-}[dd] & & & & \\
& L(\xi)  \ar@{-}[luu]  \ar@{-}[ruu]& & & & & & L(\xi)  \ar@{-}[luu]  \ar@{-}[ruu] & \\
& & & &   & & & & \\
& & & & L \ar@{-}[llluu] \ar@{-}[d] \ar@{-}[rrruu] & & & & \\
& & & & \mathbb{Q}(T) & & & & \\
}}
\]

\subsection{Résultats préparatoires}
Commençons par déterminer $\Gal(M_1/L)$ et $\Gal(M_2/L)$.
\begin{lemma}\label{unicycl}
a) Les extensions $(L(\xi ,x_{\ell,k})/L(\xi))_{\ell \in \{1,2 \},k \in (\Z/m\Z)^{*}}$ et  $(\oQ  L(x_{\ell,k})/\oQ L )_{\ell \in \{1,2 \},k \in (\Z/m\Z)^{*}}$ sont cycliques de degré $m$ et lin\'{e}airement disjointes dans leur ensemble \footnote{Rappelons que des extensions galoisiennes finies $M_1/K, \dots ,M_n/K$ sont linéairement disjointes dans leur ensemble si, pour toute partition $(I, J)$ de $\llbracket 1, n \rrbracket$, le compositum des $M_{i}$ ($i \in I$) est linéairement disjoint sur $K$ du compositum des $M_j$ ($j \in J$).}. Ainsi, pour $\ell \in \{1,2\}$, l'extension $M_{\ell}/L(\xi)$ est galoisienne de groupe $(\Z/m\Z)^{\varphi(m)}$, où $\varphi$ est l'indicatrice d'Euler. De plus, pour $\ell \in \{1,2\}$, les idéaux premiers de $\overline{\mathcal{O}}$ ramifiés dans $\oQ M_{\ell}$ sont les $\mathcal{P}_{\ell,k}$ ($k \in (\Z/m\Z)^{*}$).
\vskip 1 mm
\noindent
b) Soit $\ell \in \{1,2\}$. L'extension $M_{\ell}/L$ est galoisienne de degré $m^{\varphi(m)}\varphi(m)$. De plus, pour tout $\omega \in \Gal(M_\ell/L)$, il existe un unique  $v(\omega) \in (\Z/m\Z)^{*}$ tel que $\omega(\xi)=\xi^{v(\omega)}$ et, pour tout $k \in (\Z/m\Z)^{*}$, il existe un unique $s_{k} \in \Z/m\Z$ tel que $\omega(x_{\ell,k})=\xi^{s_{k}}x_{\ell,v(\omega)k}$. L'application
$$f : \left \{ \begin{array} {ccc}
\Gal (M_\ell/L) & \rightarrow & (\Z/m\Z)^* \times (\Z/m\Z)^{\varphi(m)} \\
\omega & \mapsto & (v(\omega), (s_k)_{k \in (\Z/m\Z)^{*}})\\
\end{array} \right.$$
est en fait bijective.
\end{lemma}
\begin{proof}[Preuve]
a) Soient $\ell \in \{1,2\}$ et $k \in (\Z/m\Z)^{*}$. Du fait que $X^{m}-(T+1+(-1)^{\ell}\sqrt{T^{2}+1}-\xi^{k}) \in \overline{\mathcal{O}}[X]$ est d'Eisenstein pour l'id\'{e}al premier $\mathcal{P}_{\ell,k}$, on obtient que $\oQ L(x_{\ell,k})/\oQ L$ et $L(\xi,x_{\ell,k})/L(\xi)$ sont de degr\'{e} $m$. Par la th\'{e}orie de Kummer, $\oQ L(x_{\ell,k})/\oQ L$ et $L(\xi,x_{\ell,k})/L(\xi)$ sont cycliques. De plus, les sous-extensions de $L(\xi,x_{\ell,k})/L(\xi)$ (resp. $\oQ L(x_{\ell,k})/\oQ L$) sont les $(L(\xi,x_{\ell,k}^{a})/L(\xi))_{a \mid m}$ (resp. $(\oQ L(x_{\ell,k}^{a})/\oQ L)_{a \mid m}$). Pour tout diviseur $a$ de $m$, diff\'{e}rent de $m$, le seul id\'{e}al premier de $\overline{\mathcal{O}}$ ramifi\'{e} dans $
\oQ L(x_{\ell,k}^{a})/ \oQ L$ est $\mathcal{P}_{\ell,k}$. Comme les idéaux premiers $\mathcal{P}_{\ell,k}$ ($\ell \in \{1,2\}$ et $k \in (\Z/m\Z)^{*}$) sont distincts (voir lemme \ref{ramggg}), le lemme d'Abhyankar fournit la conclusion voulue.
\vskip 2mm
\noindent
b) Le corps $M_\ell$ est le corps de décomposition sur $L$ du polynôme 
$ \prod_{k \in (\Z/m\Z)^{*}} (X^{m}-x_{\ell,k}^m) \in L[X],$ ce qui montre que $M_\ell/L$ est galoisienne. De plus, puisque $L/\Q$ est régulière, $\Q(T,\xi)$ et $L$ sont linéairement disjoints sur $\Q(T)$. On a donc $[L(\xi):L]= \varphi(m)$ et, par le a), $[M_\ell:L]= m^{\varphi(m)} \varphi(m)$. Il est alors clair que l'application $f$ est bien définie et que celle-ci est injective. Pour des raisons de cardinalité, elle est aussi surjective.
\end{proof}
La proposition suivante est inspirée de \cite[Chapter III, Theorem 4.3]{MM18} et permet de déterminer le groupe de Galois de $N_{\ell}/L$.
\begin{proposition}\label{thcycl}
Soit $\ell \in \{1,2\}$.
\vskip 0.5 mm
\noindent
a) Le polyn\^{o}me $h_{\ell}(X)$ est dans $L[X]$ et est irr\'{e}ductible sur $\oQ L$.
\vskip 0.5mm
\noindent
b) L'extension $N_{\ell}/L$ est galoisienne de groupe $\Z/m\Z$. De plus, un g\'{e}n\'{e}rateur de $\Gal(N_{\ell}/L)$ est donn\'{e} par $\gamma_{\ell}$, où $\gamma_{\ell}$ vérifie $\gamma_{\ell}(z_{\ell,l})=z_{\ell,l+1}$ pour tout $l \in \N$.
\end{proposition}
\begin{proof}[Preuve]
Ecrivons 
\[
h_{\ell}(X)=X^{m}+\sum_{\iota=1}^{m}(-1)^{\iota}s_{\ell,\iota}X^{m-\iota}
\] dans $\overline{\Q(T)}[X]$. Pour tout $ \varsigma \in \llbracket 1, m \rrbracket$, notons
\[
q_{\ell,\varsigma}= \sum_{l=1}^{m}z_{\ell,l}^{\varsigma} 
.\] Des identit\'{e}s de Newton, pour tout $\varsigma \in \llbracket 1, m \rrbracket$,  on a 
\begin{equation}
q_{\ell,\varsigma}+\sum_{\iota=1}^{\varsigma-1}(-1)^{\iota}s_{\ell,\iota}q_{\ell,\varsigma-\iota}+(-1)^{\varsigma}\varsigma s_{\ell,\varsigma}=0. 
\label{idnew}
\end{equation}

\vskip 2mm
\noindent
a) Pour $\varsigma \in \llbracket 1, m \rrbracket $, on a $q_{\ell,\varsigma}\in \mathcal{O}[\xi]$. En effet, soit $\varsigma \in \llbracket 1 ,m \rrbracket$. On a
{\everymath={\displaystyle}
$$
\begin{array}{lllll}
q_{\ell,\varsigma}&=& \sum_{l=1}^{m}\left(\sum_{j \in (\Z/m\Z)^{*}}\xi^{lj}y_{\ell,j}\right)^{\varsigma} & = &\sum_{l=1}^{m}\prod_{\lambda=1}^{\varsigma} \left( \sum_{j_{\lambda} \in (\Z/m\Z)^{*}} \xi^{lj_{\lambda}}  y_{\ell,j_{\lambda}} \right)  \\
	&&&=& \sum_{l=1}^{m} \left(\sum_{(j_{1},\dots,j_{\varsigma}) \in ((\Z/m\Z)^{*})^{\varsigma}} \left(\prod_{\lambda=1}^{\varsigma}\xi^{lj_{\lambda}}y_{\ell,j_{\lambda}}\right) \right)\\
	&&&=& \sum_{(j_{1},\dots,j_{\varsigma}) \in ((\Z/m\Z)^{*})^{\varsigma}} \left(\sum_{l=1}^{m}\prod_{\lambda=1}^{\varsigma}\xi^{lj_{\lambda}}\right) \left(\prod_{\lambda=1}^{\varsigma}y_{\ell,j_{\lambda}}\right)\\
	&&&=&\sum_{(j_{1},\dots,j_{\varsigma}) \in ((\Z/m\Z)^{*})^{\varsigma}} \left(\sum_{l=1}^{m}\xi^{l(\sum_{\lambda=1}^{\varsigma}j_{\lambda})}\right) \left(\prod_{\lambda=1}^{\varsigma}y_{\ell,j_{\lambda}} \right).
\end{array} 
$$}

\noindent
Du fait que, pour tout $o \in \Z$, la somme $\sum_{l=1}^{m} \xi^{lo}$ est \'{e}gale \`{a} $m$ si $m$ divise $o$ et est \'{e}gale \`{a} $0$ sinon, les seuls termes non nuls dans la somme ci-dessus sont ceux v\'{e}rifiant $\sum_{\lambda=1}^{\varsigma}j_{\lambda} \equiv 0 \pmod m$. On obtient alors
$$
q_{\ell,\varsigma}  = \sum_{\substack{(j_{1},\dots,j_{\varsigma}) \in ((\Z/m\Z)^{*})^{\varsigma} \\ \sum_{\lambda=1}^{\varsigma}j_{\lambda}\equiv 0 \pmod m} } m \prod_{\lambda=1}^{\varsigma}\left(\prod_{k \in (\Z/m\Z)^{*}} x_{\ell,k}^{r_{m}(j_{\lambda}/k)}\right)= \sum_{\substack{(j_{1},\dots,j_{\varsigma}) \in ((\Z/m\Z)^{*})^{\varsigma} \\ \sum_{\lambda=1}^{\varsigma}j_{\lambda}\equiv 0 \pmod m }} m \left(\prod_{k \in (\Z/m\Z)^{*}} x_{\ell,k}^{\sum_{\lambda=1}^{\varsigma}r_{m}(j_{\lambda}/k)}\right).
$$
Comme, pour tout $k \in (\Z/m\Z)^{*}$, on a
\begin{equation}
\sum_{\lambda=1}^{\varsigma}r_{m}(j_{\lambda}/k)\equiv \left( \sum_{\lambda=1}^{\varsigma}j_{\lambda} \right)/k \equiv 0 \pmod m,
\label{rr}
\end{equation} les termes de $q_{\ell,\varsigma}$ appartiennent bien \`{a}  $ \mathcal{O}[\xi]$.

Par cons\'{e}quent, les $s_{\ell,\iota}$ ($\iota \in \llbracket 1,m \rrbracket$) sont dans $\mathcal{O}[\xi]$ en utilisant la relation \eqref{idnew}.

De plus, par le lemme \ref{unicycl}, tout $\omega \in \Gal(L(\xi)/L)$ se rel\`{e}ve en un unique $\widetilde{\omega} \in \Gal(M_\ell/L)$ tel que, pour tout $k \in (\Z/m\Z)^{*}$, on ait $\widetilde{\omega}(x_{\ell,k})=x_{\ell,v(\omega)k}$, o\`{u} $v(\omega) $ est l'unique \'{e}l\'{e}ment de $(\Z/m\Z)^{*}$ v\'{e}rifiant $\omega(\xi)=\xi^{v(\omega)}$. Tout $\widetilde{\omega}$ de cette forme fixe les $z_{\ell,l}$ car, pour tout $l \in \llbracket 1,m \rrbracket$, on a
{\everymath={\displaystyle}
$$
\begin{array}{lllll}
\widetilde{\omega}(z_{\ell,l})&= &\widetilde{\omega}\bigg(\sum_{j \in (\Z/m\Z)^{*}}\xi^{lj}\prod_{k \in (\Z/m\Z)^{*}}x_{\ell,k}^{r_{m}(j/k)}\bigg) &= & \sum_{j \in (\Z/m\Z)^{*}}\xi^{ljv(\omega)}\prod_{k \in (\Z/m\Z)^{*}}x_{\ell,v(\omega)k}^{r_{m}(j/k)}\\
	&&&=&\sum_{j \in (\Z/m\Z)^{*}}\xi^{lj}\prod_{k \in (\Z/m\Z)^{*}}x_{\ell,k}^{r_{m}(j/k)}\\
	&&&=&z_{\ell, l}.
\end{array} 
$$}

\noindent
Des r\'{e}sultats pr\'{e}c\'{e}dents, on d\'{e}duit que les $s_{\ell,\iota}$ sont dans $\mathcal{O}$.

Montrons maintenant que $h_{\ell}(X)$ est irr\'{e}ductible sur $\oQ L$. Pour cela, il suffit de v\'{e}rifier que c'est un polyn\^{o}me d'Eisenstein pour l'id\'{e}al premier $\mathcal{P}_{\ell,1}$. En utilisant l'\'{e}quation \eqref{rr}, on voit que $q_{\ell,\varsigma} \in \mathcal{P}_{\ell,1}$ pour tout $\varsigma$.  Par cons\'{e}quent, en utilisant l'identit\'{e} \eqref{idnew}, on a $s_{\ell,\iota} \in \mathcal{P}_{\ell,1}$ pour tout $\iota \in \llbracket 1,m\rrbracket$. Donc, il reste \`{a} montrer que $s_{\ell,m} \notin \mathcal{P}_{\ell,1}^{2}$. En utilisant l'identit\'{e} \eqref{idnew} avec $\varsigma=m$ et le fait que $s_{\ell,\iota}q_{\ell,m-\iota} \in\mathcal{P}_{\ell,1}^{2}$ pour tout $\iota \in \llbracket 1,m-1 \rrbracket$, il suffit de v\'{e}rifier que $ q_{\ell,m} \notin \mathcal{P}_{\ell,1}^{2}$.

Pour cela, notons $v$ la valuation de $\oQ L$ associ\'{e}e \`{a} $\mathcal{P}_{\ell,1}$. Pour tout $(j_1, \dots, j_m) \in ((\Z/m\Z)^{*})^{m}$ tel que $\sum_{\lambda=1}^{m}j_{\lambda}\equiv 0 \pmod m$, on a 
\[
v \bigg( x_{\ell,k}^{\sum_{\lambda=1}^{m}r_{m}(j_{\lambda}/k)}\bigg)=0
\] pour tout $k \in (\Z/m\Z)^{*} \setminus \{1 \}$ par le lemme \ref{ramggg}. Par conséquent, on a
\begin{align*}
v\left(m \bigg(\prod_{k \in (\Z/m\Z)^{*}} x_{\ell,k}^{\sum_{\lambda=1}^{m}r_{m}(j_{\lambda}/k)}\bigg)\right)&=v\bigg(x_{\ell,1}^{\sum_{\lambda=1}^{m}r_{m}(j_{\lambda})}\bigg)\\
&=v\left(\Big( T+1+(-1)^{\ell}\sqrt{T^{2}+1}-\xi\Big)^{\frac{\sum_{\lambda=1}^{m}r_{m}(j_{\lambda})}{m}}\right)\\
&=\frac{1}{m}\Big(\sum_{\lambda=1}^{m}r_{m}(j_{\lambda})\Big).
\end{align*}
En remarquant que la valuation ci-dessus est \'{e}gale \`{a} $1$ si $(j_1, \dots, j_m)=(1,\dots,1)$ et sup\'{e}rieure ou \'{e}gale \`{a} $2$ sinon, on obtient $v(q_{\ell,m})=1$, ce qui achève la démonstration du a).
\vskip 2mm
\noindent
b) Pour tout $j \in (\Z/m\Z)^{*}$, l'extension $L(\xi,y_{\ell,j})/L(\xi)$ est cyclique de degr\'{e} $m$. En effet, $X^{m}-\prod_{k \in (\Z/m\Z)^{*}}x_{\ell,k}^{mr_{m}(j/k)} \in L(\xi)[X]$ annule $y_{\ell,j}$ et est d'Eisenstein pour l'id\'{e}al premier $\mathcal{P}_{\ell,j}$. On conclut donc par la th\'{e}orie de Kummer.

Maintenant, si $F_{\ell}$ désigne le compositum des $L(\xi,y_{\ell,j})$ ($j\in (\Z/m\Z)^{*}$), on a $F_{\ell}=L(\xi,y_{\ell,1})$. En effet, soit $j \in (\Z/m\Z)^{*}$. Pour tout $k \in (\Z/m\Z)^{*}$, il existe $o_{j,k} \in m\Z$ tel que $ r_{m}(j)r_{m}(1/k)=r_{m}(j/k)+o_{j,k}$.
On a alors \[y_{\ell,1}^{r_{m}(j)}=y_{\ell,j}\bigg(\prod_{k \in (\Z/m\Z)^{*}}\bigg(T+1+(-1)^{\ell}\sqrt{T^{2}+1}-\xi^{k}\bigg)^{o_{j,k}/m}\bigg),
\] donc $L(\xi,y_{\ell,j}) \subset L(\xi,y_{\ell,1})$. Par cons\'{e}quent, on a $F_{\ell}=L(\xi,y_{\ell,1})$.

De plus, $F_{\ell}/L$ est galoisienne. En effet, $L(\xi)/L$ est clairement galoisienne et, d'après les deux paragraphes précédents, $F_\ell/L(\xi)$ l'est aussi. Par conséquent, il suffit de montrer que tout élément de $\Gal(L(\xi)/L)$ se relève en un élément de $\Aut(F_\ell/L)$. Fixons pour cela $\omega \in \Gal(L(\xi)/L)$. Notons comme précédemment $\widetilde{\omega}$ l'unique relèvement de $\omega$ à $M_\ell$ tel que, pour tout $k \in (\Z/m\Z)^{*}$, on ait $\widetilde{\omega}(x_{\ell,k})=x_{\ell,v(\omega)k}$, o\`{u} $v(\omega) $ est l'unique \'{e}l\'{e}ment de $(\Z/m\Z)^{*}$ v\'{e}rifiant $\omega(\xi)=\xi^{v(\omega)}$. Pour tout $j \in (\Z/m\Z)^{*}$, on a $\widetilde{\omega}(y_{\ell,j})=y_{\ell,v(\omega)j}$, ce qui montre que la restriction de $\widetilde{\omega}$ à $F_\ell$ est bien un élément de $\Aut(F_\ell/L)$.

En outre, $N_{\ell}/L$ est cyclique de degr\'{e} $m$. En effet, d'après le a), $N_{\ell}/L$ est de degr\'{e} $m$. Notons $\Omega=\{ \res^{M/L}_{F_\ell/L}(\widetilde{\omega})\, \mid \, \omega \in \Gal(L(\xi)/L) \}$, o\`{u} $\widetilde{\omega}$ est d\'{e}fini dans le paragraphe précédent . Vu que $v(\omega \omega')=v(\omega)v(\omega')$ pour tous $\omega,\omega' \in  \Gal(L(\xi)/L)$, l'ensemble $\Omega$ est un sous-groupe de $\Gal(F_\ell/L)$. Par restriction, $\Omega$ est isomorphe  \`{a} $\Gal(L(\xi)/L)$. De plus, la preuve du a) montre que $\Omega$ fixe chaque élément du sous-corps $N_{\ell}$ de $F_{\ell}$. Pour des raisons de degrés, on a forc\'{e}ment $N_{\ell}=F_{\ell}^{\Omega}$. De plus, par le a), les $L$-conjugu\'{e}s de $z_{\ell,1}$ sont les $z_{\ell,l}$ ($l \in \llbracket 1 , m \rrbracket$). Ils sont fix\'{e}s par $\Omega$ par la preuve du a), donc ils sont dans $N_{\ell}$. On en d\'{e}duit que $N_{\ell}/L$ est galoisienne de degr\'{e} $m$. On remarque enfin que $L(\xi) \cap N_{\ell}=L(\xi) \cap F_{\ell}^{\Omega}=L(\xi)^{\Gal(L(\xi)/L)}=L$ et donc $F_{\ell}=N_{\ell}(\xi)$. On en d\'{e}duit $\Gal(N_{\ell}/L) \simeq \Gal(F_{\ell}/L(\xi)) \simeq \Z/m\Z$.

Enfin, $\gamma_{\ell}$ d\'{e}finit bien un g\'{e}n\'{e}rateur de $\Gal(N_{\ell}/L)$. En effet, avec les notations du lemme \ref{unicycl}, soit $\omega$ l'unique élément de $\Gal(M_\ell/L)$ défini par $v(\omega)=1$, $s_{1}=1$ et $s_{k}=0$ pour tout $k \in (\Z/m\Z)^{*} \setminus \{1\}$.  On v\'{e}rifie que $\omega (z_{\ell,l})=z_{\ell,l+1}$ pour tout $l \in \llbracket 1, m \rrbracket$. Ainsi $\gamma_{\ell}=\res^{M_\ell/L}_{N_{\ell}/L}(\omega)$ est un élément de $\Gal(N_{\ell}/L)$ d'ordre $m$. 
\end{proof}
Nous relevons maintenant le générateur $\beta$ de $\Gal(L/\Q(T))$ en un élément $\chi$ de $\Aut(M_1M_2/\Q(T))$.
\begin{lemma}\label{lemcrt}
Il existe un automorphisme $\chi$ de $M_1 M_2$ prolongeant $\beta$, fixant $\xi$ et tel que $\chi(x_{2,k})=x_{1,k}$ et $\chi(x_{1,k})=x_{2,k}$ pour tout $k \in (\Z/m\Z)^{*}$. En conséquence, $\chi(z_{1,l})=z_{2,l}$ et $\chi(z_{2,l})=z_{1,l}$ pour tout $l \in \llbracket 1 , m \rrbracket$. De plus, l'extension $M_1M_2/\Q(T)$ est galoisienne.
\end{lemma}
\begin{proof}[Preuve]
On remarque d'abord que l'on peut \'{e}tendre $\beta$ en un automorphisme de $L(\xi)$ fixant $\xi$.
Ensuite, on étend $\beta$ en un isomorphisme $M_2 \rightarrow M_1$ tel que $\beta(x_{2,k})=x_{1,k}$ pour tout $k \in (\Z/m\Z)^{*}$. En effet, pour tout $k \in (\Z/m\Z)^{*}$, on a $(X^m-(T+1+\sqrt{T^{2}+1}-\xi^{k}))^{\beta}= X^m-(T+1-\sqrt{T^{2}+1}-\xi^{k})$ et le polyn\^{o}me minimal de $x_{2,k}$ (resp. $x_{1,k}$) sur le compositum des corps $L(\xi,x_{2,k'})$ (resp. $L(\xi,x_{1,k'})$), pour $k'  \in (\Z/m\Z)^* \setminus \{k\}$, est $X^m-(T+1+\sqrt{T^{2}+1}-\xi^{k})$ (resp. $X^m-(T+1-\sqrt{T^{2}+1}-\xi^{k})$), d'après le lemme \ref{unicycl}. 
Enfin, on \'{e}tend $\beta$ en un automorphisme $\chi$ de $M_1M_2$ tel que $\chi(x_{1,k})=x_{2,k}$ pour tout $k \in (\Z/m\Z)^{*}$. En effet, pour tout $k \in (\Z/m\Z)^{*}$, on a $(X^m-(T+1-\sqrt{T^{2}+1}-\xi^{k}))^{\beta}= X^m-(T+1+\sqrt{T^{2}+1}-\xi^{k})$ et le polyn\^{o}me minimal de $x_{1,k}$ (resp. $x_{2,k}$) sur le compositum des corps $M_2$ et $L(\xi,x_{1,k'})$ (resp. $M_1$ et $L(\xi,x_{2,k'})$), pour $k'  \in (\Z/m\Z)^* \setminus \{k\}$, est $X^m-(T+1-\sqrt{T^{2}+1}-\xi^{k})$ (resp. $X^m-(T+1+\sqrt{T^{2}+1}-\xi^{k})$). 
\end{proof}

Nous déterminons enfin le groupe de Galois de $F/\Q(T)$.
\begin{proposition}\label{pdcour} 
a) L'extension $F/\Q(T)$ est galoisienne et $\Q$-régulière.
\vskip 0.5mm
\noindent
b) On a $
\Gal(F/L)= \{ \mathcal{L}_{a,b} \mid (a,b) \in (\Z/m\Z)^{2} \},
$ où  $\mathcal{L}_{a,b}(z_{1,l})=z_{1,l+a}$ et $\mathcal{L}_{a,b}(z_{2,l})=z_{2,l+b}$ pour tout $(a,b) \in \Z/m\Z \times \Z/m\Z$ et tout $l \in \N$.
\vskip 0.5mm
\noindent
c) Il existe $g \in \Gal(F/\Q(T))$, d'ordre $2$, vérifiant $\res^{F/\Q(T)}_{L/\Q(T)}(g)=\beta$ et $g(z_{1,l})=z_{2,l}$ (resp.
$g(z_{2,l})=z_{1,l}$) pour tout $l \in \llbracket 1 ,m \rrbracket $. De plus, on a \begin{equation}\label{descwfom}
\Gal(F/\Q(T))= \{ \mathcal{L}_{a,b} \circ g^{\epsilon} \mid (a,b) \in (\Z/m\Z)^2 \, , \epsilon \in \{0,1\} \},
\end{equation}
qui est isomorphe \`a $\Z/m\Z \wr \Z/2\Z$ via
$$\psi : \left \{ \begin{array} {ccc}
\Z/m\Z \wr \Z/2\Z& \rightarrow & \Gal(F/\Q(T))  \\
((a,b),\epsilon) & \mapsto & \mathcal{L}_{a,b} \circ g^{\epsilon} \\
\end{array} \right..$$
\end{proposition}
\begin{proof}[Preuve] 
a) Montrons tout d'abord que l'on a $F \cap \overline{\Q}=\Q$. Pour cela, notons que $M_1 \oQ$ et $M_2 \oQ$ sont linéairement disjoints sur $L \oQ$ en vertu du lemme \ref{unicycl}. En particulier, $N_{1} \oQ$ et $N_{2} \oQ$ le sont également. On a donc $[N_{1} N_{2} \oQ : L \oQ]=m^2$ par la proposition \ref{thcycl}. Ainsi, on a $[N_{1} N_{2} \oQ :  \oQ(T)]=2m^2$, ce qui montre bien $F \cap \overline{\Q}=\Q$ puisque $[F : \Q(T)] \leq 2m^2$. Au passage, on a montré que $N_{1}$ et $N_{2}$ étaient linéairement disjoints sur $L$. Ensuite, notons $\chi$ un automorphisme de $M_1M_2$ comme dans le lemme \ref{lemcrt}. On a $\chi(z_{1,l})=z_{2,l}$ et $\chi(z_{2,l})=z_{1,l}$ pour tout $l \in \N$, donc la restriction de $\chi$ à $F$, que l'on note $g$, est un automorphisme de $F$. Ainsi, $F/\Q(T)$ est galoisienne.  
\vskip 0.5 mm
\noindent
b) Soient $\gamma_{1}$ et $\gamma_{2}$ les g\'{e}n\'{e}rateurs de $\Gal(N_{1}/L)$ et $\Gal(N_{2}/L)$ de la proposition \ref{thcycl}. Puisque $N_{1}$ et $N_{2}$ sont linéairement disjoints sur $L$, on obtient un isomorphisme
$$\varphi : \left \{ \begin{array} {ccc}
\Z/m\Z \times \Z/m\Z& \rightarrow & \Gal(F/L)  \\
(a,b) & \mapsto & (\gamma_{1}^{ a},\gamma_{2}^{b})=\mathcal{L}_{a,b} \\
\end{array} \right..$$
c) Tout d'abord, on a 
\begin{equation}\label{equa11}
\Gal(F /\Q(T))=\Gal(F/L)  \rtimes \langle g \rangle.
\end{equation} En effet, notons que $\Gal(F/L)$ est bien un sous-groupe distingué de $\Gal(F /\Q(T))$ et que l'on a $\Gal(F/L)  \cap \langle g \rangle = \{{\rm{Id}}_{F}\}$ puisque $g$ est d'ordre 2 et prolonge $\beta$. Le b) donne alors \eqref{descwfom}.

Considérons maintenant l'application $\psi$ de la proposition. Par le b) et l'égalité \eqref{equa11}, l'application $\psi$ est bijective. Pour montrer que $\psi$ est un isomorphisme, il suffit de v\'{e}rifier que $\varphi$ pr\'{e}serve l'action de $\Z/2\Z$. A cet effet, pour $ (a,b) \in \Z/m\Z \times \Z/m\Z$, on a
$$ \begin{array} {ccccccc}
g \circ \varphi(a,b) \circ g^{-1}(z_{1,1}) &=& g \circ \varphi(a,b)(z_{2,1})
										   & = &g (\gamma_{2}^{b}(z_{2,1})) 	&				   = & g (z_{2,1+b}) \\
				&&& & &
										   = &  z_{1,1+b}\\
										  && &&& = &\gamma_{1}^{b}(z_{1,1})\\
										  && && &= &\varphi(b,a)(z_{1,1}) \\
										   &&&&& = &\varphi((a,b)^{\overline{1}})(z_{1,1})
\end{array}$$
et, de m\^eme, $g \circ \varphi(a,b) \circ g^{-1}(z_{2,1})=\varphi((a,b)^{\overline{1}})(z_{2,1})$.
\end{proof}

\subsection{Résultat principal}\label{real2} 
\begin{theorem}\label{princ}
On se donne un produit semi-direct $\Z/m\Z \rtimes \Z/2\Z$, uniquement déterminé par l'image $d$ de $1 \in \Z/2\Z$ dans $\Aut(\Z/m\Z)=(\Z/m\Z)^{*}$. Pour tout $\delta \in \llbracket 1 ,m \rrbracket$, posons \[
v_{\delta}(d)=\sum_{l=1}^{m}z_{2,l}z_{1,-dl+\delta}.
\] 
\vskip 0.5 mm
\noindent
a) L'extension $E_{d}/\Q(T)=\Q(T,\sqrt{T^2+1}+v_{0}(d))/\Q(T)$ est une $(\Z/m\Z \rtimes \Z/2\Z)$-extension $\Q$-r\'{e}guli\`{e}re. De plus,  on a $E_{d}=L(v_0(d))$ et les $\Q(T)$-conjugu\'{e}s de $\sqrt{T^2+1}+v_{0}(d)$ sont les
\[
(-1)^{\epsilon}\sqrt{T^{2}+1}+v_{\delta}(d^{(-1)^\epsilon}), \quad  (\delta,\epsilon) \in \llbracket 1 ,m \rrbracket \times \{0,1\}.  
\]
b) Le groupe $\Gal(E_{d}/L)$ est engendré par un automorphisme $r$ vérifiant  $r(v_{\delta}(d))=v_{\delta+1}(d)$ pour tout $\delta \in \llbracket 1 ,m \rrbracket$.  En posant $s=\res^{F/\Q(T)}_{E_{d}/\Q(T)}(g)$, où $g$ est défini dans la proposition \ref{pdcour}, on a \[\Gal(E_{d}/\Q(T))=\{ r^{\delta} \circ s^{\epsilon} \mid \delta \in \llbracket 1, m \rrbracket \, , \epsilon \in \{0,1 \} \},\] qui est isomorphe à $\Z/m\Z \rtimes \Z/2\Z$ via $(\delta,\epsilon) \mapsto r^{\delta} \circ s^{\epsilon}$.
\end{theorem}
\begin{proof}[Preuve] 
D'après \cite[Lemma 16.4.3]{FJ08}, l'application suivante est un épimorphisme :
\[
\begin{array}{ccccc}
\alpha &:&\Z/m\Z \wr \Z/2\Z & \rightarrow & \Z/m\Z \rtimes \Z/2\Z     \\
         &  &        \left((a,b),\eta \right)        & \mapsto     &  (a+db,\eta). 
\end{array}
\]
Son noyau est
$
\mathrm{Ker}(\alpha)=\{ (-da,a) \mid a \in \Z/m\Z \}
$. Ci-dessous, on utilise les notations $\mathcal{L}_{a,b}$ ($(a,b) \in \Z/m\Z \times \Z/m\Z$) et $\psi$ de la proposition \ref{pdcour}.

Le sous-corps de $F$ fix\'{e} par $\psi(\mathrm{Ker}(\alpha))$ est $L(v_0(d))$. De plus, $L(v_0(d))/L$ est une $\Z/m\Z$-extension de groupe de Galois engendré par l'automorphisme $r$ vérifiant $r(v_{\delta}(d))=v_{\delta+1}(d)$ pour tout $\delta \in \llbracket 1,m \rrbracket$. En effet, en remarquant que 
$\psi(\mathrm{Ker}(\alpha))=\{ \mathcal{L}_{-da,a} \mid a \in \Z/m\Z \} \subset \Gal(F/L),$ on obtient 
$$\Gal(F^{\psi(\mathrm{Ker}(\alpha))}/L)=\Gal(F/L)/\psi(\mathrm{Ker}(\alpha))=\{ \res^{F/L}_{F^{\psi(\mathrm{Ker}(\alpha))}/L}(\mathcal{L}_{\delta,0})  \mid \delta \in \llbracket 1, m \rrbracket \}=\langle r \rangle,$$ où $r=\res^{F/L}_{F^{\psi(\mathrm{Ker}(\alpha))}/L}(\mathcal{L}_{1,0})$. On voit que
$v_{\delta}(d)=\mathcal{L}_{\delta,0}(v_0(d))$ pour tout $\delta \in \llbracket 1,m \rrbracket$, donc les $v_{\delta}(d)$ sont $L$-conjugués. Il est imm\'{e}diat que l'on a $v_0(d) \in F^{\psi (\mathrm{Ker}(\alpha))}$ et donc $L(v_0(d)) \subset F^{\psi(\mathrm{Ker}(\alpha))} $. Pour l'inclusion inverse, puisque $[F^{\psi(\mathrm{Ker}(\alpha))}:L]=m$, il suffit de montrer que les $v_{\delta}(d)$ sont deux \`{a} deux distincts. Pour cela, soient $\delta \not \equiv t \pmod m$. On a
\begin{align*}
v_{\delta}(d)-v_t(d) &=\sum_{l=1}^{m}z_{2,l}(z_{1,-dl+\delta}-z_{1,-dl+t})\\
		&=\sum_{l=1}^{m}z_{2,l}\left(\sum_{j \in (\Z/m\Z)^{*}}\xi^{-dlj}(\xi^{j\delta}-\xi^{jt})y_{1,j}\right)\\
		&=\sum_{l=1}^{m}\left( \sum_{j^{'} \in (\Z/m\Z)^{*}} \xi^{lj^{'}}y_{2,j'} \right)\left(\sum_{j \in (\Z/m\Z)^{*}}\xi^{-dlj}(\xi^{j\delta}-\xi^{jt})y_{1,j}\right)\\
		&=\sum_{l=1}^{m}\sum_{j,j^{'} \in (\Z/m\Z)^{*}}\xi^{l(-dj+j^{'})}(\xi^{j\delta}-\xi^{jt})\prod_{k \in (\Z/m\Z)^{*}}x_{2,k}^{r_{m}(j^{'}/k)}x_{1,k}^{r_{m}(j/k)}\\
		&=\sum_{j \in (\Z/m\Z)^{*}}m(\xi^{j\delta}-\xi^{jt}) \prod_{k \in (\Z/m\Z)^{*}}x_{2,k}^{r_{m}(dj/k)}x_{1,k}^{r_{m}(j/k)}.
\end{align*}
La derni\`{e}re \'{e}galit\'{e} vient du fait que $\sum_{l=1}^{m}\xi^{lo}$ est \'{e}gal \`{a} $0$ si $o \not \equiv 0 \pmod m$ et $m$ sinon. Soit $\mathcal{P}$ un id\'{e}al premier au dessus de $\mathcal{P}_{1,1}$ dans $\oQ M_1M_2/\oQ L$. Pour $j \neq 1$ (resp. $j=1$), on a $r_{m}(j/1) \geq 2$ (resp. $r_{m}(j/1)=1$). Donc on a
\begin{equation}\label{valsui} 
0 <v(v_{\delta}(d)-v_{t}(d))=v\left(m(\xi^{\delta}-\xi^{t}) \prod_{k \in (\Z/m\Z)^{*}}x_{2,k}^{r_{m}(d/k)}x_{1,k}^{r_{m}(1/k)}\right) < \infty,
 \end{equation} o\`{u} $v$ est la valuation associ\'{e}e \`{a} $\mathcal{P}$. Par cons\'{e}quent, on a $v_{\delta}(d)-v_t(d) \neq 0 $.

Maintenant, on a $E_d=L(v_0(d))$ et $\Gal(E_d/\Q(T))=\{ r^{\delta} \circ s^{\epsilon} \mid \delta \in \llbracket 1, m \rrbracket \, , \epsilon \in \{0,1 \} \}.$ En effet, pour $\delta \in \llbracket 1, m \rrbracket$ et $\epsilon \in \{0,1 \}$, posons $w_{\epsilon,\delta}=(-1)^{\epsilon}\sqrt{T^{2}+1}+v_{\delta}(d^{(-1)^\epsilon})$. Il est déjà clair que l'on a $E_{d}= \Q(T)(w_{0,0}) \subset L(v_0(d))$. De plus, pour tout $\epsilon \in \{0, 1\}$ et tout $\delta \in \llbracket 1, m \rrbracket$, on a $w_{\epsilon, \delta} = \mathcal{L}_{\delta,0} \circ g^{\epsilon}(w_{0,0})$. Il suffit donc de montrer que les $w_{\epsilon,\delta}$ sont deux \`{a} deux distincts. Soient donc $(\epsilon,\delta)$ et $(\epsilon^{'},\delta^{'})$ tels que $w_{\epsilon,\delta}=w_{\epsilon^{'},\delta^{'}}$. Supposons dans un premier temps $\epsilon =  \epsilon'$ et $\delta \neq \delta^{'}$. On a donc $0= v_\delta(d^{(-1)^\epsilon}) - v_{\delta'}(d^{(-1)^\epsilon})$ et l'on aboutit à une contradiction comme dans le paragraphe précédent. Supposons maintenant $\epsilon \not=\epsilon'$. On a alors
$$0= ((-1)^\epsilon - (-1)^{\epsilon'})\sqrt{T^2+1} + v_\delta(d^{(-1)^\epsilon}) - v_{\delta'}(d^{(-1)^{\epsilon'}}).$$
Comme dans le paragraphe précédent, soit $\mathcal{P}$ un idéal premier au dessus de $\mathcal{P}_{1,1}$ dans $\oQ M_1 M_2/\oQ L$. Si $v$ désigne à nouveau la valuation associée à $\mathcal{P}$, on a $v(v_\delta(d^{(-1)^\epsilon}) - v_{\delta'}(d^{(-1)^{\epsilon'}}))>0$. Or $v(\sqrt{T^2+1})=0$. En effet, si $\sqrt{T^2+1}$ était dans $\mathcal{P}$, alors $T-i$ ou $T+i$ le serait aussi. D'après le lemme \ref{ramggg}, $i$ ou $-i$ serait alors égal à $-(1-\xi)/2 + 1/(2(1-\xi))$, ce qui est impossible. Puisque $(-1)^\epsilon - (-1)^{\epsilon'} \not=0$, on en déduit $v(0)=v(w_{\epsilon,\delta}-w_{\epsilon^{'},\delta^{'}})=0$, une contradiction.
\end{proof}

\subsection{Corollaires}

Les trois énoncés ci-dessous s'obtiennent à partir du théorème précédent en considérant successivement les cas suivants :
\begin{enumerate}[1)]
\item $m$ arbitraire et $d=-1$,
\item $m$ égal à une puissance de 2 et $d=(m/2)-1$,
\item $m$ égal à une puissance de 2 et $d=(m/2)+1$.
\end{enumerate}  

\begin{corollary}\label{princ1}
L'extension $\Q(T,\sqrt{T^{2}+1}+\sum_{l=1}^{m}z_{2,l}z_{1,l})$  est une $D_{2m}$-extension $\Q$-r\'{e}guli\`{e}re. De plus, les $\Q(T)$-conjugu\'{e}s de $\sqrt{T^{2}+1}+\sum_{l=1}^{m}z_{2,l}z_{1,l}$ sont les
\[
(-1)^{\epsilon}\sqrt{T^{2}+1}+\sum_{l=1}^{m}z_{2,l}z_{1,l+\delta} ,\quad ( \delta, \epsilon) \in  \llbracket 1, m \rrbracket \times  \{0,1\} .
\]
\end{corollary}

\begin{corollary}\label{princ2}
Soit $n \geq 3$. L'extension $\Q(T,\sqrt{T^{2}+1}+\sum_{l=1}^{2^{n-1}}z_{2,l}z_{1,-(2^{n-2}-1)l})$ est une $QD_{2^n}$-extension $\Q$-r\'{e}guli\`{e}re et les $\Q(T)$-conjugu\'{e}s de $\sqrt{T^{2}+1}+\sum_{l=1}^{2^{n-1}}z_{2,l}z_{1,-(2^{n-2}-1)l}$ sont les
\[
(-1)^{\epsilon}\sqrt{T^{2}+1}+\sum_{l=1}^{2^{n-1}}z_{2,l}z_{1,-(2^{n-2}-1)^{(-1)^\epsilon}l+\delta}, \quad (\delta,\epsilon) \in \llbracket 1, 2^{n-1} \rrbracket \times  \{0,1\}.
\]
\end{corollary}

\begin{corollary}\label{princ3}
Soit $n \geq 3$. L'extension $\Q(T,\sqrt{T^{2}+1}+\sum_{l=1}^{2^{n-1}}z_{2,l}z_{1,-(2^{n-2}+1)l})$ est une $M_{2^n}$-extension $\Q$-r\'{e}guli\`{e}re et les $\Q(T)$-conjugu\'{e}s de $\sqrt{T^{2}+1}+\sum_{l=1}^{2^{n-1}}z_{2,l}z_{1,-(2^{n-2}+1)l}$ sont les
\[
(-1)^{\epsilon}\sqrt{T^{2}+1}+\sum_{l=1}^{2^{n-1}}z_{2,l}z_{1,-(2^{n-2}+1)^{(-1)^\epsilon}l+\delta}, \quad (\delta, \epsilon) \in  \llbracket 1, 2^{n-1} \rrbracket \times \{0,1\}.
\]
\end{corollary}

\section{Groupes de quaternions généralisés}

\subsection{Notations}\label{notprin11}
On considère la $\Z/2\Z$-extension $L/\Q(T)=\Q(T)(\sqrt{T^{2}+1})/\Q(T)$. On se donne $n \geq 3$ et on note $\xi=\exp(2 \pi i/2^{n-1})$. Pour tout $k \in (\Z/2^{n-1}\Z)^{*}$, on pose $s_{k}=-(1-\xi^{k})/2+1/(2(1-\xi^{k}))$. Pour tout $k \in (\Z/2^{n-1}\Z)^{*}$ et tout $\ell \in \{1,2\}$, on choisit une racine $2^{n-1}$-i\`{e}me $x_{\ell,k}$ de $T+1+(-1)^{\ell}\sqrt{T^{2}+1}-\xi^{k}$ dans $\overline{\Q(T)}$. Pour $ \ell \in \{1,2\}$, on note $M_\ell$ le compositum des corps $L(\xi, x_{\ell, k})$ $(k \in (\Z/2^{n-1}\Z)^*$). Pour tout $\ell \in \{1,2\}$ et tout $l \in \N$, on pose
\[
z_{\ell,l}=\sum_{j \in (\Z/2^{n-1}\Z)^{*}}\xi^{lj}\prod_{k \in (\Z/2^{n-1}\Z)^{*}}x_{\ell,k}^{r_{2^{n-1}}(j/k)},
\] o\`{u} $r_{2^{n-1}}:(\Z/2^{n-1}\Z)^{*}\rightarrow \llbracket 0, 2^{n-1}-1\rrbracket$ envoie $k \in (\Z/2^{n-1}\Z)^{*}$ sur son unique repr\'{e}sentant modulo $2^{n-1}$. De l'égalité $\xi^{2^{n-2}}=-1=(-1)^{j}$ pour tout entier impair $j$, on déduit
\begin{equation}\label{egalzz}
z_{1,l+2^{n-2}}=\sum_{j \in (\Z/2^{n-1}\Z)^{*}}\xi^{2^{n-2}j}\xi^{lj}\prod_{k \in (\Z/2^{n-1}\Z)^{*}}x_{1,k}^{r_{2^{n-1}}(j/k)}=-z_{1,l}.
\end{equation}
Comme dans le théorème \ref{princ}, on pose $v_{\delta}(-1)=\sum_{l=1}^{2^{n-1}}z_{2,l}z_{1,l+\delta}$ pour tout $\delta \in \llbracket 0, 2^{n-1}-1\rrbracket$. Pour simplifier, on écrira $v_\delta$ au lieu de $v_\delta (-1)$. Les $v_{\delta}$ sont deux à deux distincts et sont conjugués sur $L$ par le théorème \ref{princ}. De plus, pour tout $\delta \in \N$, \eqref{egalzz} entraîne
\begin{equation}\label{ppac}
v_{\delta+2^{n-2}}=-v_{\delta}.
\end{equation} 

On note $E/\Q(T)$ la $D_{2^n}$-extension $\Q$-régulière fournie par le corollaire \ref{princ1}.
\subsection{Réalisations régulières explicites}
Dans cette partie, nous construisons une extension galoisienne $\Q$-régulière explicite de $\Q(T)$ de groupe $Q_{2^n}$ (voir théorème \ref{quaternion}). 
 
Commençons par déterminer $\Gal(HE/\Q(T))$. 
\begin{proposition}\label{Deltann}
 L'extension $HE/\Q(T)$ est une $\Delta_n$-extension $\Q$-régulière. De plus, le groupe $\Gal(HE/\Q(T))$ est engendré par deux éléments $\rho$ et $\sigma$ vérifiant les propriétés suivantes :
\vskip 0.5mm
\noindent
- $\rho(\mu_4)=\mu_4$ et $\rho(v_{\delta})=v_{\delta+1}$ pour tout $\delta \in \llbracket 1, 2^{n-1} \rrbracket$,
\vskip 0.5mm
\noindent
- $\sigma(\mu_4)=\mu_1$ et $\sigma(v_{0})=v_{0}$.
\vskip 0.5mm
\noindent
En outre, $\rho$ et $\sigma$ vérifient la présentation \eqref{presdelt} et $\sigma^{2}$ fixe chaque élément de $E$.
\end{proposition}
\begin{proof}[Preuve] L'extension $L/\Q(T)=E^{\langle r \rangle}/\Q(T)$, où $r$ est défini dans le théorème \ref{princ}, se plonge dans la $\Z/4\Z$-extension $\Q$-régulière $H/\Q(T)$. La proposition \ref{regDelta} assure alors que $HE/\Q(T)$ est une $\Delta_n$-extension $\Q$-régulière. Par la proposition \ref{propgen}, il existe un unique relèvement $\rho$ de $r$ dans $\Gal(H E/H	)$ et celui-ci vérifie ce qui suit : pour tout relèvement $\widetilde{\tau}$ de $\tau$ dans $\Delta_{n}$, les éléments $\rho$ et $\widetilde{\tau}$ engendrent $\Delta_{n}$ et vérifient \eqref{presdelt}. Fixons maintenant un tel $\widetilde{\tau}$ et considérons sa restriction à $E$. Il est alors clair qu'il existe un entier $t$ tel que  ${\rm{res}}^{H E/\Q(T)}_{E/\Q(T)}(\widetilde{\tau}) = r^t s$, où $s$ est défini dans le théorème \ref{princ}, et que l'on a
${\rm{res}}^{H E/\Q(T)}_{E/\Q(T)}(\widetilde{\tau})(v_{0}) = r^t(v_{0})$. La composée $\sigma=\rho^{-t}\widetilde{\tau}$ est alors un relev\'e de $\tau$ à $H E$ qui fixe $v_{0}$. Aussi, non seulement $\rho$ et $\sigma$ vérifient bien les deux propriétés de l'énoncé, mais ils vérifient aussi \eqref{presdelt} (voir proposition \ref{propgen}). Pour le dernier point, il suffit de remarquer que $E=L(v_{0})$ (voir théorème \ref{princ}), que $\res^{H E/\Q(T)}_{L/\Q(T)}(\sigma^2)=\mathrm{Id}_{L}$ et que $\sigma^{2}(v_{0})=v_{0}$.
\end{proof}

\begin{theorem}\label{quaternion}
L'extension $\Q(T,\sqrt{T^2+1}+(\mu_4 + v_0)^2)/\Q(T)$
est une $Q_{2^{n}}$-extension $\Q$-r\'{e}guli\`{e}re. De plus, les $\Q(T)$-conjugu\'{e}s de $\sqrt{T^2+1}+(\mu_4 + v_0)^2$ sont les 
\[(-1)^a \sqrt{T^2+1}+(\mu_a + v_{\delta})^2 , \quad (a,\delta) \in \llbracket 1,4\rrbracket \times \llbracket 0, 2^{n-2}-1\rrbracket.
\] 
\end{theorem}
\begin{proof}[Preuve]
La proposition \ref{Deltann} montre que $H E/\Q(T)$ est une $\Delta_{n}$-extension $\Q$-régulière. Considérons les générateurs $\rho$ et $\sigma$ de $\Gal(H E/\Q(T))$ définis dans cette même proposition.

On a $H E=L(\mu_4 +v_0)$. En effet, par construction, le corps $H E$ est le compositum sur $L$ des corps $L(\mu_4)=H$ et $L(v_0)=E$, qui sont linéairement disjoints sur $L$ (voir proposition \ref{propgen}). Comme $L(\mu_4)/L$ et $L(v_0)/L$ sont galoisiennes, on obtient bien que $\mu_4 +v_0$ est un élément primitif de $H E$ sur $L$ (voir, par exemple, \cite[Proposition 3.5.5]{Wei09}).

De plus, $L((\mu_4 +v_0)^2)/\Q(T)$ est une $Q_{2^{n}}$-extension $\Q$-r\'{e}guli\`{e}re. En effet, le polyn\^{o}me $X^{2}-(\mu_4 +v_0)^{2}$ annule $\mu_4 +v_0$, donc $[H E:L((\mu_4 +v_0)^2)] \leq 2$. D'après \eqref{ppac}, on a 
$$\rho^{2^{n-2}}\circ \sigma ^{2}(\mu_4 +v_0)=\rho^{2^{n-2}}(-\mu_4 + v_0) = -\mu_4 + v_{2^{n-2}} = -(\mu_4 +v_0).$$
L'élément $\rho^{2^{n-2}}\sigma^{2}$  correspond \`{a} $(2^{n-2},2)$ dans $\Delta_{n}$, qui est d'ordre $2$. Ainsi, pour des raisons de degré, on obtient $L((\mu_4 +v_0)^2) =(H E)^{\langle \rho^{2^{n-2}}\sigma^{2}\rangle}$ et $\Gal(L((\mu_4 +v_0)^2)/\Q(T)) = Q_{2^n}$.

Enfin, on a $\Q(T,\sqrt{T^2+1}+(\mu_4 + v_0)^2)=L((\mu_4 +v_0)^2)$. En effet, pour tout $(a,\delta) \in \llbracket 1,4\rrbracket \times \llbracket 0, 2^{n-2}-1\rrbracket$, posons $w_{a,\delta}=(-1)^a \sqrt{T^2+1}+(\mu_a + v_{\delta})^2$. ll suffit de montrer que $w_{0,0}$ est un \'{e}l\'{e}ment primitif de $L((\mu_4 +v_0)^2)$ sur $\Q(T)$. Il est clair que $w_{0,0} \in L((\mu_4 +v_0)^2)$. Pour tout $(a,\delta) \in \llbracket 1,4\rrbracket \times \llbracket 0, 2^{n-2}-1\rrbracket$, on a $w_{a,\delta}= \rho^{\delta} \circ \sigma^{a}(w_{0,0})$. Il suffit donc de montrer que les $w_{a,\delta}$ sont deux \`{a} deux distincts. Soient $(a,\delta)$ et $(b,t)$ deux \'{e}l\'{e}ments distincts de $\llbracket 1,4\rrbracket \times \llbracket 0, 2^{n-2}-1\rrbracket$ tels que $w_{a,\delta}=w_{b,t}$. On remarque tout d'abord que l'on a $v_{\delta} \neq - v_{t}$. En effet, si ce n'était pas le cas, alors on aurait $v_{\delta}=-v_{t}=v_{t+2^{n-2}}$ (par l'égalité \eqref{ppac}) et donc $\delta=t+2^{n-2}$ car les $v_\delta$ sont deux à deux distincts comme mentionné dans le \S\ref{notprin11}, une contradiction. De l'égalité $w_{a,\delta}=w_{b,t}$, un petit calcul montre que
\begin{equation}\label{malssss}
2(\mu_{a}v_{\delta}-\mu_{b}v_{t})=v_{t}^{2}-v_{\delta}^{2}+((-1)^b-(-1)^a)(T+1)\sqrt{T^2+1} \in E \subset (H E)^{\langle \sigma^2 \rangle},
\end{equation} où l'inclusion $E \subset (H E)^{\langle \sigma^2 \rangle}$ vient de la proposition \ref{Deltann}. Ainsi, $
\mu_{a}v_{\delta}-\mu_{b}v_{t}=\sigma^{2}(\mu_{a}v_{\delta}-\mu_{b}v_{t})=-(\mu_{a}v_{\delta}-\mu_{b}v_{t})
$ et donc $\mu_{a}v_{\delta}-\mu_{b}v_{t}=0$. On distingue maintenant deux cas.

Supposons d'abord $a$ et $b$ de même parité. L'équation \eqref{malssss} donne $v_{t}^{2}-v_{\delta}^{2}=0$. Ainsi on a $v_{t}=v_{\delta}$ car $v_{t} \neq - v_{\delta}$. En conséquence, on obtient $t=\delta$. De l'égalité $\mu_{a}v_{\delta}-\mu_{b}v_{t}=0$, on déduit $a=b$, une contradiction.

Supposons maintenant $a$ et $b$ de parité différente. En remarquant que
\[
v_{\delta}^2-v_{t}^2=(v_{\delta}-v_t)(v_{\delta}+v_t)=(v_{\delta}-v_t)(v_{\delta}-v_{t+2^{n-2}}),
\] l'égalité \eqref{malssss} donne 
$
(v_{\delta}-v_t)(v_{\delta}-v_{t+2^{n-2}})=2(-1)^{b}(T+1)\sqrt{T^2+1}.
$
Comme précédemment, soit $\mathcal{P}$ un id\'{e}al premier de $\oQ M_1 M_2$ contenant $T+1-\sqrt{T^{2}+1}-\xi$ de valuation associée $v$. L'équation \eqref{valsui} dans la preuve du théorème \ref{princ} montre que $v(2(-1)^{b}(T+1)\sqrt{T^2+1})=v((v_{\delta}-v_t)(v_{\delta}-v_{t+2^{n-2}}))>0$. Par le lemme \ref{ramggg}, on a $v(2(-1)^{b}(T+1)\sqrt{T^2+1})=v((T+1)^2(T^2+1))=0$, une contradiction. Par conséquent $w_{a,\delta} \neq w_{b,t}$.
\end{proof}

\subsection{Réalisations explicites de $Q_{2^n}$ sur $\Q$}
Dans cette partie, nous construisons des réalisations explicites de $Q_{2^n}$ sur $\Q$ par spécialisation de la $Q_{2^n}$-extension $\Q$-régulière $\Gamma/\Q(T)$ fournie par le théorème \ref{quaternion} (voir théorème \ref{thm:spec_arith}).

\subsubsection{Bons premiers}
On donne ici une condition suffisante pour qu'un nombre premier $p$ soit un bon premier pour $\Gamma/\Q(T)$ (voir proposition \ref{bonpremier}).
\begin{lemma}\label{Bragg}
a) L'ensemble des points de branchement de $H/\Q(T)$ est contenu dans $\{ -i,i,0, \infty\}$.
\vskip 0.5mm
\noindent
b)  L'ensemble des points de branchement de $M_{1}M_{2}/\Q(T)$ est contenu dans $\{ -i,i,\infty \}\cup \{ s_{k} \mid k \in (\Z/2^{n-1}\Z)^{*} \}$.
\vskip 0.5mm
\noindent
c)  L'ensemble des points de branchement de $\Gamma/\Q(T)$ est contenu dans $\{ -i,i,0,\infty \} \cup \{ s_{k} \mid k \in (\Z/2^{n-1}\Z)^{*}\}$.
\end{lemma}
\begin{proof}[Preuve] a) Cela vient du fait que le discriminant du polynôme minimal de $\mu_{4}$ sur $\oQ (T)$ est $256 T^{4}(T^{2}+1)^{3}$.
\vskip 2mm
\noindent
b) Soit $ \lambda \in \oQ$ un point de branchement de $M_{1}M_{2} /\Q (T)$ qui n'est pas dans $ \{-i,i \}$. Par \cite[Proposition 6.2.3]{Sti09}, $\lambda$ n'est pas un point de branchement de $L/\Q(T)$. En conséquence, on peut trouver un idéal premier $\mathcal{P}$ de $\oQ L$ contenant $T- \lambda$ et se ramifiant dans $\oQ M_{1}M_{2} / \oQ L$. Par le lemme \ref{unicycl} et le lemme d'Abhyankar, il existe $\ell \in \{1,2\}$ et $k \in (\Z/2^{n-1}\Z)^{*}$ tels que $\mathcal{P}$ soit engendré par $T+1+(-1)^{\ell}\sqrt{T^{2}+1}-\xi^{k}$. D'où $\lambda=s_{k}$ par le lemme \ref{ramggg}.
\vskip 2mm
\noindent
c) Il suffit de remarquer que l'on a $\Gamma \subset  H M_1 M_2$ et d'utiliser le a) et le b). 
\end{proof}
\begin{lemma} \label{unitariser}
Soit $p$ un nombre premier impair qui ne divise pas $2^{2^{n-2}}+1$. Alors, pour tout $t \in \{i,-i \} \cup \{s_{k} \mid k \in (\Z/2^{n-1}\Z)^{*}\}$ et tout idéal premier de $\Z[\xi]$ au dessus de $p$ de valuation associée $v$, on a $v(t)=0$.
\end{lemma}
\begin{proof}[Preuve]
Le cas où $t \in \{-i,i \}$ est immédiat. Il reste à considérer le cas où $t \in \{s_{k} \mid k \in (\Z/2^{n-1}\Z)^{*}\}$. Soit donc $k \in(\Z/2^{n-1}\Z)^{*}$. On a 
\begin{equation}\label{egalsk}
s_{k}=\frac{\xi^{k}(2-\xi^{k})}{2(1-\xi^{k})}.
\end{equation} 
On remarque ensuite que le polynôme minimal de $1-\xi^{k}$ (resp. $2-\xi^{k}$) sur $\Q$ est $(X-1)^{2^{n-2}}+1$ (resp. $(X-2)^{2^{n-2}}+1$). Donc la norme de $1-\xi^{k}$ (resp. $2-\xi^{k}$) dans $\Q(\xi)/\Q$ est $2$ (resp. $2^{2^{n-2}}+1$). Par conséquent, si $v$ est comme dans l'énoncé, on a $v(x)=0$ dès que $x$ est un $\Q$-conjugué de $1-\xi^{k}$ ou $2-\xi^k$. On peut ainsi conclure en vertu de \eqref{egalsk}.
\end{proof}
\begin{lemma}\label{ssssoi}
Pour tout $x \in \Q(\xi)$, on note $N(x)$ la norme de $x$ dans $\Q(\xi)/\Q$.
Soit $p$ un nombre premier impair tel que
\vskip 0.5mm
\noindent
a) pour tout $k \not \equiv 1  \pmod {2^{n-1}}$, le nombre premier $p$ ne divise pas
$
N((1-\xi^{k})(1-\xi)+1),
$
\vskip 0.5mm
\noindent
b) pour tout $k \in (\Z/2^{n-1}\Z)^{*}$, le nombre premier $p$ ne divise pas
$
N( \xi^{k}(2-\xi) \pm 2i(1-\xi^{k})  ).
$
\vskip 0.5mm
\noindent
Alors deux points de branchement quelconques et distincts de $\Gamma/\Q(T)$ ne peuvent se rencontrer modulo $p$. Cette dernière conclusion est en particulier vraie pour $p \geq 7^{2^{n-2}}+1$.
\end{lemma}
\begin{proof}[Preuve] Par le lemme \ref{Bragg}, il suffit de vérifier que deux éléments quelconques et distincts de $\{-i,i,0,\infty\}\cup \{ s_{k} \mid k \in (\Z/2^{n-1}\Z)^{*}\}$ ne peuvent se rencontrer modulo $p$. 

D'après le lemme \ref{unitariser}, on a $v(s_{k})=0$ pour tout $k \in (\Z/2^{n-1}\Z)^{*}$ et toute valuation $v$ de $\Q(\xi)$ étendant la valuation $p$-adique. Par conséquent, $s_{k}$ et $0$ ne peuvent se rencontrer modulo $p$. Il en est de même pour $s_k$ et $\infty$.

On remarque maintenant que les $s_k$ ($k \in (\Z/2^{n-1}\Z)^{*}$) sont deux à deux conjugués sur $\Q$. Par conséquent, pour montrer qu'il n'existe pas $k \not= l$ tels que $s_k$ et $s_l$ se rencontrent modulo $p$, il suffit de le faire pour $k \not=1$ et $l=1$. Fixons donc $k \not=1$. On a
\[
s_{k}-s_{1}=\frac{\xi^{k}-\xi}{2(1-\xi^{k})(1-\xi)}((1-\xi^{k})(1-\xi)+1 ).
\]
D'après l'hypothèse du a), le nombre premier $p$ ne divise pas $N((1-\xi^{k})(1-\xi)+1)$. De plus, on montre comme dans la preuve du lemme \ref{unitariser} que $p$ ne divise, ni $N(2(1-\xi^{k})(1-\xi))$, ni $N(\xi^{k}-\xi)$. Le lemme \ref{normren} montre alors que $s_k$ et $s_1$ ne se rencontrent pas modulo $p$.

Ensuite, pour tout $k \in (\Z/2^{n-1}\Z)^{*}$, on a
\[
s_{k}\pm i =\frac{\xi^{k}(2-\xi^{k})\pm 2i (1-\xi^{k})}{2(1-\xi^{k})}
\]
et on montre comme ci-dessus que $s_k$ et $\pm i$ ne peuvent se rencontrer modulo $p$.

Pour le dernier point, on remarque que les valeurs absolues des normes de $a)$ et $b)$ sont inférieures ou égales à $7^{2^{n-2}}$.
\end{proof}

\begin{lemma}\label{ertfg}
Aucun nombre premier impair $p$ n'est verticalement ramifié dans $\Gamma/\Q(T)$.
\end{lemma}
\begin{proof}[Preuve]
 Il suffit de travailler au dessus du localisé $R=\Z[T]_{p\Z[T]}$ de $\Z[T]$ en $p\Z[T]$. Remarquons que $\Gamma \subset HM_{1}M_{2}$ et notons $C$ la clôture intégrale de $R$ dans $HM_{1}M_{2}$.

Tout d'abord, le discriminant du polynôme minimal de $\mu_{4}$ sur $\Q(T)$ est $256T^4(T^{2}+1)^{3} \not \in p \Z[T]$, donc $pR$ est non ramifié dans $H/\Q(T)$. 

Soient maintenant $\ell_{0} \in \{1,2\}$ et $k_{0} \in (\Z/2^{n-1}\Z)^{*}$. Les $\Q(T)$-conjugués de $x_{\ell_{0},k_{0}}$ sont parmi les $\xi^{l}x_{\ell,k}$ où $l \in \llbracket 1, 2^{n-1} \rrbracket$, $\ell \in \{1,2\}$ et $k \in (\Z/2^{n-1}\Z)^{*}$. En posant
\[
\eta_{(l,\ell,k),(l^{'},\ell^{'},k^{'})}=\xi^{l}x_{\ell,k}-\xi^{l^{'}}x_{\ell^{'},k^{'}}
\] 
pour tous $(l,\ell,k) \neq (l^{'},\ell^{'},k^{'})$, on voit que le discriminant du polynôme minimal de $x_{\ell_{0},k_{0}}$ sur $\Q(T)$ divise 
$\prod_{(l,\ell,k)\neq(l^{'},\ell^{'},k^{'})}\eta_{(l,\ell,k),(l^{'},\ell^{'},k^{'})}$ dans $C$.

Soient $(l,\ell,k) \neq (l^{'},\ell^{'},k^{'})$. Pour simplifier, on pose $\eta=\eta_{(l,\ell,k),(l^{'},\ell^{'},k^{'})}$. Supposons tout d'abord $\ell \neq \ell^{'}$. Par définition des $x_{\ell, k}$, on a 
$\eta r=4T^{2}+4-(\xi^{k}-\xi^{k^{'}})^{2} \in C,$ où
$$
r=(((-1)^l-(-1)^{l'})\sqrt{T^2+1}-(\xi^k-\xi^{k'}) ) \prod_{s=0}^{2^{n-2}}((\xi^{l}x_{\ell,k})^{2^s}-(\xi^{l^{'}}x_{\ell^{'},k^{'}})^{2^s}) \in C.
$$
On peut encore multiplier 
$4T^{2}+4-(\xi^{k}-\xi^{k^{'}})^{2}$ par ses $\Q(T)$-conjugués pour obtenir qu'il existe un $r' \in C$ tel que 
$\eta r ' = 4^{u}T^{v}+a,$ où $u \in \mathbb{N}$, $v \in \mathbb{N}$ et $a \in \Z[T]$ de degré inférieur ou égal à $v-1$.

Supposons maintenant $\ell = \ell{'}$ et $k=k{'}$. Dans ce cas, on a 
$$\eta x_{\ell,k}^{2^{n-1}-1}=\xi^{l}(1-\xi^{l{'}-l})(T+1+(-1)^{\ell}\sqrt{T^{2}+1}-\xi^{k}).$$ 
En multipliant ce dernier élément par $T+1-(-1)^{\ell}\sqrt{T^{2}+1}-\xi^{k} \in C$, on obtient qu'il existe un $r'$ dans $C$ tel que $\eta r '= \xi^{l}(1-\xi^{l^{'}-l})(2(1-\xi^{k})T+(1-\xi)^{2}-1)$. 
Or, la norme de $\xi$ dans $\Q(\xi)/\Q$ vaut $\pm 1$ et, comme déjà vu, la norme de $1-\xi^{k}$ dans cette même extension est une puissance de 2 pour tout $k \not \equiv 0  \pmod {2^{n-1}}$. On peut encore multiplier 
$$\xi^{l}(1-\xi^{l^{'}-l})(2(1-\xi^{k})T+(1-\xi)^{2}-1)$$ par ses $\Q(T)$-conjugués pour obtenir qu'il existe un $r'' \in C$ tel que 
$\eta r '' = \pm 2^{u}T^{v}+a,$ où $u \in \mathbb{N}$, $v \in \mathbb{N}$ et $a \in \Z[T]$ de degré inférieur ou égal à $v-1$.

Supposons enfin $\ell=\ell{'}$ et $k \neq k{'}$. Dans ce cas, il existe un élément $r$ de $C$ tel $\eta r=2^u$ avec $u \in \mathbb{N}$. En effet, le produit de $\eta$ et  
$\prod_{s=0}^{2^{n-2}}((\xi^{l}x_{\ell,k})^{2^s}-(\xi^{l^{'}}x_{\ell,k^{'}})^{2^s})$ vaut $-\xi^{k}(1-\xi^{k'-k})$, dont la norme est une puissance de $2$.

En considérant les trois cas ci-dessus, il existe donc un élément $r$ de $C$ tel que le produit de $r$ et du discriminant du polynôme minimal de $x_{\ell_{0},k_{0}}$ sur $\Q(T)$ soit de la forme $2^{u}T^{v}+a \in \Z[T]$ avec $u \in \N$, $v \in \N$ et $a \in \Z_{(p)}[T]$ de degré inférieur ou égal à $v-1$. Par conséquent, ce discriminant n'est pas dans $pR$. En particulier, $pR$ n'est pas ramifié dans $\Q(T)(x_{\ell_{0},k_{0}})$. De plus, $pR$ n'est pas ramifié dans $ \Q(T, \xi)/\Q(T)$. Le lemme d'Abhyankar permet alors de conclure.
\end{proof} 
\begin{proposition}\label{bonpremier} Soit $p$ un nombre premier impair qui vérifie les conditions du lemme \ref{ssssoi}. Alors $p$ est un bon premier pour $\Gamma/\Q(T)$.
\end{proposition}
\begin{proof}[Preuve]
D'après les lemmes \ref{ssssoi} et \ref{ertfg}, $p$ ne vérifie aucune des conditions 2) et 3) de la définition \ref{bonpremiers}. De plus, $p$ ne vérifie pas la condition 1). Quant à la condition 4), il suffit de voir que le corps engendré par les points de branchement de $\Gamma/\Q(T)$ est contenu dans $\Q(\xi)$.
\end{proof}
\subsubsection{Invariant canonique de l'inertie }
Dans ce qui suit, $\Lambda$ et $\Sigma$ sont les générateurs de $Q_{2^n}$ de la présentation \eqref{presquat}. Rappelons que les classes de conjugaison non triviales de $Q_{2^n}$ sont
\vskip 0.5mm
\noindent
- les $\mathcal{A}_{j}=\{\Lambda^{j},\Lambda^{-j}\}$ où $j \in \llbracket 1, 2^{n-2}-1\rrbracket$ et $\mathcal{A}_{2^{n-2}}=\{\Lambda^{2^{n-2}} \}$,
\vskip 0.5mm
\noindent
- $\mathcal{B}=\{\Lambda^{2j}\Sigma \mid j \in \llbracket 0, 2^{n-2}-1\rrbracket\}$,
\vskip 0.5mm
\noindent
- $\mathcal{C}=\{\Lambda^{2j+1}\Sigma \mid j \in \llbracket0, 2^{n-2}-1\rrbracket\}$.

\begin{lemma}\label{classedin}
La classe canonique de l'inertie du point de branchement $i$ de $\Gamma/\Q(T)$ est $\mathcal{B}$ ou $\mathcal{C}$. De plus, il existe un $k \in (\Z/2^{n-1} \Z)^{*}$ et un $j  \in \llbracket 1, 2^{n-2}\rrbracket$ impair tels que la classe canonique de l'inertie du point de branchement $s_k$ de $\Gamma/\Q(T)$ soit $\mathcal{A}_j$,
\end{lemma}
\begin{proof}[Preuve]
Pour tout $t \in \{ -i,i,0, \infty \} \cup \{ s_{k} \mid k \in (\Z/2^{n-1}\Z)^{*}\}$, notons $C_t$ la classe canonique de l'inertie de $t$ dans $\Gamma/\Q(T)$. Puisque $i$ est point de branchement de $L/\Q(T)$, la classe $C_i$ n'est pas contenue dans $\Gal(\Gamma/L)=\langle \Lambda \rangle$. De plus, on vérifie facilement que l'on a $\mathcal{B}^3= \mathcal{B}$ et $\mathcal{C}^3= \mathcal{C}$. Ainsi, par le {\it Branch Cycle Lemma} (voir \cite{Fri77} et \cite[Lemma 2.8]{Vol96}), on a $C_i=C_{-i}$.

Ensuite, puisque les points de branchement de $L/\Q(T)$ sont $i$ et $-i$, la classe $C_0$ est de type $\mathcal{A}$ ou triviale et il en est de même des classes $C_{s_k}$ (voir lemme \ref{ramggg}). Supposons que toutes ces classes de conjugaison soient triviales ou de la forme $\mathcal{A}_j$ avec $j$ pair. Puisque l'ensemble des points de branchement de $\Gamma/\Q(T)$ est contenu dans $\{ -i,i,0, \infty \} \cup \{ s_{k} \mid k \in (\Z/2^{n-1}\Z)^{*}\}$ (voir lemme \ref{Bragg}), le théorème d'existence de Riemann fournit $(g_i, g_{-i}, g_0, g_\infty, (g_{s_k})_k) \in C_i  \times C_{-i} \times C_0 \times C_\infty \times (C_{s_k})_k$ tel que $\langle g_i, g_{-i}, g_0, g_\infty, \{g_k \mid k\} \rangle = Q_{2^n}$ et $g_i  \cdot g_{-i} \cdot g_0 \cdot g_\infty \cdot \prod_{k} g_k =1$. Notons $g_i = \Lambda^{a_i} \Sigma$, $g_{-i} = \Lambda^{a_{-i}} \Sigma$, avec $a_i$ et $a_{-i}$ de même parité, $g_0= \Lambda^{2a_0}$, $g_\infty = \Lambda^{a_\infty}$ et $g_{s_k} = \Lambda^{2a_k}$ pour tout $k$. On obtient alors
$1 = g_i  \cdot g_{-i} \cdot g_0 \cdot g_\infty \cdot \prod_{k} g_k = \Lambda^v$
pour un certain entier $v$ de même parité que $a_i + a_{-i} + 2a_0 + a_\infty + \sum_k 2 a_k$. Par conséquent, $a_\infty$ est pair. La condition $\langle g_i, g_{-i}, g_0, g_\infty, \{g_k \mid k\} \rangle = Q_{2^n}$ entraîne alors $Q_{2^n} = \langle \Lambda^2, \Sigma \rangle$ (si $C_i = C_{-i} = \mathcal{B}$) ou $Q_{2^n} = \langle \Lambda^2, \Lambda \Sigma \rangle$ (si $C_i = C_{-i} = \mathcal{C}$), ce qui est impossible. En effet, $\langle \Lambda^2 \rangle$ est un sous-groupe distingué d'ordre $2^{n-2}$ et $\Sigma^2=\Lambda^{2^{n-2}} \in \langle \Lambda^2 \rangle $ (resp. $(\Lambda\Sigma)^2=\Lambda^{2^{n-2}} \in \langle \Lambda^2 \rangle $), donc $| \langle \Lambda^2, \Sigma \rangle|$ (resp. $| \langle \Lambda^2, \Lambda \Sigma \rangle|$) est $2^{n-1}$ alors que $Q_{2^n}$ est d'ordre $2^n$. Ainsi, soit la deuxième partie du lemme est vraie, soit l'indice de ramification de $\langle T \rangle$ dans $\Gamma \oQ / \oQ(T)$ vaut $2^{n-1}$. Or, cette dernière conclusion est impossible car $0$ n'est pas un point de branchement de $M_1 M_2 /\Q(T)$ (voir lemme \ref{Bragg}) et l'indice de ramification de $\langle T \rangle $ dans $H \oQ / \oQ (T)$ vaut au plus $2$. 
\end{proof}

\subsubsection{Théorème principal}
Notons $m(X) = \prod_{k \in (\Z/2^{n-1}\Z)^*} (X-s_k) \in \Q[X]$. Par le lemme \ref{unitariser}, les dénominateurs des coefficients de $m(X)$ sont des puissances de 2.

\begin{lemma}\label{divpremier}
 Pour $p \equiv 1 \pmod {2^{n-1}}$ premier, il existe $t \in \Z$ tel que $v_{p}(m(t))>0$.
 \end{lemma}
\begin{proof}[Preuve]
Notons $O$ la clôture intégrale du localisé $\Z_{[2]}$ de $\Z$ en la partie multiplicative $\{2^j \mid j \geq 0 \}$ dans $\Q(\xi)$. Soit $\mathcal{P}$ un idéal premier de $O$ au dessus de $p\Z_{[2]}$. Puisque $p \equiv 1 \pmod{2^{n-1}}$, le degré résiduel $f(\mathcal{P}/p\Z_{[2]})$ est égal à $1$, c'est-à-dire $O/\mathcal{P}=\Z_{[2]}/p\Z_{[2]}$. Il existe donc $t \in \Z_{[2]}$ tel que $m(t)=0$ modulo $p\Z_{[2]}$. On peut en fait choisir $t$ dans $\Z$, ce qui conclut la démonstration.
\end{proof}
\begin{theorem} \label{thm:spec_arith}
Soient $p$ et $q$ deux nombres premiers distincts ne divisant pas $2^{2^{n-2}}+1$, qui vérifient les conditions du lemme \ref{ssssoi} et tels que $p \equiv 1 \pmod {2^{n-1}}$ et $q \equiv 1 \pmod {4}$. Soit $t_0 \in \N$ vérifiant $v_{p}(m(t_{0})) = 1$ et $v_{q}(t_{0}^{2}+1)=1$. Alors la spécialisation $\Gamma_{t}/\Q$ de $\Gamma/\Q(T)$ en $t$ est galoisienne de groupe $Q_{2^{n}}$ pour tout $t \equiv t_0 \pmod {p^2q^2}$.
\end{theorem}
\begin{proof}[Preuve]
Notons tout d'abord que $t_0$ comme dans l'énoncé existe. En effet, puisque $p \equiv 1 \pmod{2^{n-1}}$, il existe $t_{0,p}$ dans $\Z$ tel que $v_{p}(m(t_{0,p}))>0$ (voir lemme \ref{divpremier}). De plus, d'après l'hypothèse $q \equiv 1 \pmod{4}$, il existe $t_{0,q}$ dans $\Z$ tel que $v_{q}(t_{0,q}^2+1)>0$. D'après \cite[Lemma 3.12]{Leg16}, quitte à changer $t_{0,p}$ et $t_{0,q}$, on peut supposer $v_{p}(m(t_{0,p}))=1$ et $v_{q}(t_{0,q}^2+1)=1$. Le théorème chinois fournit alors un $t_0 \in \Z$ tel que $v_{p}(t_0-t_{0,p}) \geq 2$ et $v_{q}(t_0-t_{0,q}) \geq 2$. Enfin, \cite[Remark 2.11]{Leg16} montre que $t_0$ vérifie les conditions de l'énoncé.  

Fixons maintenant $t \equiv t_0 \pmod {p^2q^2}$. D'après les hypothèses $v_{p}(m(t_{0})) = 1$ et $v_{q}(t_{0}^{2}+1)=1$, on a
$v_{p}(m(t)) = 1$ et $v_{q}(t^{2}+1)=1$. En particulier, $t$ n'est pas un point de branchement de $\Gamma/\Q(T)$ (voir lemme \ref{Bragg}). Par \cite[Lemma 2.5]{Leg16}, on obtient que $t$ et $s_1$ (resp. $t$ et $i$) se rencontrent modulo $p$ (resp. $q$). En outre, d'après le lemme \ref{bonpremier}, les nombres premiers $p$ et $q$ sont de bons premiers pour $\Gamma/\Q(T)$. Enfin, $p$ unitarise $s_1$ par le lemme \ref{unitariser} et il est clair que $q$ unitarise $i$. Si $k$ désigne l'élément de $(\Z/2^{n-1}\Z)^{*}$ fourni par le lemme \ref{classedin}, le théorème \ref{spinth} assure alors que $\Gal(\Gamma_t/\Q)$ contient un élément de la classe canonique de l'inertie $C_{s_k}$ (resp. $C_i$) du point de branchement $s_k$ (resp. $i$) de $\Gamma/\Q(T)$. Par le lemme \ref{classedin}, on obtient que $\Gal(\Gamma_t/\Q)$ contient un élément de $\mathcal{A}_j$ pour un certain $j \in \llbracket 1, 2^{n-2}\rrbracket$ impair et, soit un élément de $\mathcal{B}$, soit un élément de $\mathcal{C}$. On obtient alors $\Gal(\Gamma_t/\Q)=Q_{2^n}$, ce qui achève la démonstration.
\end{proof}
Pour conclure ce texte, on détermine pour  $n=3$ un exemple explicite de $p$, $q$ et $t_0$ comme dans le théorème. D'abord, on remarque que $m(X)=X^2+X/2+5/8$. Ensuite, on s'assure que $p=53$ et $q=61$ vérifient les conditions du théorème. Enfin, des exemples de $t_0$ sont $804$,
$865$ et $1758$. Nous laissons au lecteur intéressé le soin de donner davantage d'exemples numériques. 
\bibliography{Biblio2}
\bibliographystyle{alpha}
\end{document}